\documentclass[11pt]{article}

\usepackage{authblk}
\usepackage{url} 
\usepackage{amsmath}
\usepackage{tikz}
\usepackage{tikz-cd}
\usepackage{amssymb}
\usepackage{amsthm}
\usepackage{tikz-cd}
\usepackage{multirow}
\usepackage{float}
\usepackage{color}
\usepackage{soul}
\usepackage[framemethod=tikz]{mdframed}

\newtheorem{theorem}{Theorem}[section]
\newtheorem{definition}{Definition}[section]

\usepackage[margin=2.5cm]{geometry}

\newcommand{\pp}[2]{\frac{\partial #1}{\partial #2}}

\DeclareMathOperator{\diff}{d}
\newcommand{\ed}{\ensuremath{\mathrm{d}}}
\newcommand{\ddiv}{\ensuremath{\mathrm{div}}} 
\newcommand{\mc}[1]{\ensuremath{\mathcal{#1}}}
\newcommand{\mb}[1]{\ensuremath{\mathbf{#1}}}
\newcommand{\mr}[1]{\ensuremath{\mathrm{#1}}}

\newcommand{\jump}[1]{\left[\!\left[ #1 \right]\!\right]}

\newcommand{\ie}[0] {\emph{i.e.}~}

\title{Compatible finite element spaces for geophysical fluid dynamics}
\author{A. Natale, J. Shipton and C. J. Cotter }
\affil{Department of Mathematics, Imperial College London, London, SW7 2AZ, UK}
\date{\today}

\begin{document}
\maketitle

\begin{abstract}
  Compatible finite elements provide a framework for preserving
  important structures in equations of geophysical fluid dynamics, and
  are becoming important in their use for building atmosphere and
  ocean models. We survey the application of compatible finite
  element spaces to geophysical fluid dynamics, including the
  application to the nonlinear rotating shallow water equations, and
  the three-dimensional compressible Euler equations. We summarise
  analytic results about dispersion relations and conservation
  properties, and present new results on approximation properties in
  three dimensions on the sphere, and on hydrostatic balance
  properties.
\end{abstract}


\section{Introduction}

Finite element methods have recently become a popular
discretisation approach for numerical weather prediction, mostly using
spectral elements or Discontinuous Galerkin methods
\cite{fournier2004spectral,thomas2005ncar,dennis2011cam,kelly2012continuous,kelly2013implicit,marras2013simulations,brdar2013comparison,bao2015horizontally};
the use of these methods is reviewed in
\cite{marras2015review}. Finite element methods provide the
opportunity to use more general grids in numerical weather prediction,
either to improve load balancing in massively parallel computation, or
to facilitate adaptive mesh refinement. They also allow the
development of higher order discretisations.  Compatible finite
element methods are built from finite element spaces that have
differential operators that map from one space to another.  They have
a long history in both numerical analysis (see \cite{boffi2013mixed}
for a summary of contributions) and applications, including
aerodynamics, structural mechanics, electromagnetism and porous media
flows. Until recently, they have not been considered much in the
numerical weather prediction context, although the lowest
Raviart-Thomas element (denoted $RT_0$) has been proposed for ocean
modelling \cite{walters1998robust}, and the $RT_0$ and lowest order
Brezzi-Douglas-Marini element (denoted $BDM_1$) were both analysed in
\cite{rostand2008raviart} as part of the quest for a mixed finite
element pair with good dispersion properties when applied to
geophysical fluid dynamics. The realisation in \cite{CoSh2012} that
compatible finite element methods have steady geostrophic modes when
applied to the linear shallow water equations, combined with other
properties that make them analogous to the Arakawa C-grid staggered
finite difference method has led to an effort to develop these methods
for numerical weather prediction, including the Gung Ho dynamical core
project in the UK.

In this paper, we review recent work on the development of compatible finite element methods for numerical weather prediction, concentrating on theoretical results about stability, accuracy and conservation properties. In particular, we restrict ourselves to primal-only formulations, \emph{i.e.}\ involving a single mesh, on simplicial or tensor product elements, although the use of primal-dual formulation has also been explored in the context of atmosphere simulations using arbitrarily polygonal meshes \cite{thuburn2015primal}. 
Within this survey, we also include some new results, listed below.
\begin{itemize}
\item Inclusion of the consideration of inertial modes in the analysis
  of numerical solutions of the linear rotating shallow water equations on
  the plane.
\item An extension of recent results on approximation theory to
  triangular prism elements, including approximation properties when
  solving problems in an ``atmosphere-shaped'' spherical shell domain.
  These are listed below.
  \begin{enumerate}
\item For arbitrary non-affine prismatic meshes, there is a loss of
  consistency of at least 1 degree for $H(\mr{curl})$ elements and 
  for $H(\mr{div})$ elements (which are used for velocity in this framework),
  and 2 degrees for $L^2$ elements (which are used for
  density/pressure in this framework). The consistency loss increases
  with the polynomial order.
\item If the prismatic elements are arranged in columns, but with
  arbitrary slopes on top and bottom faces of each element, then the
  loss of consistency only affects the polynomial spaces in the horizontal direction. 
\item If the mesh can be obtained from a global
  continuously-differentiable transformation from a smooth
  triangulation of the sphere extruded into uniform layers, then the
  expected consistency order of the chosen finite element spaces is
  obtained. This would arise in the case of a terrain-following mesh
  with smooth terrain, for example.
\item The latter result also applies to a cubed sphere mesh under the
  same conditions.
\end{enumerate}
\item An analysis of hydrostatic balance in the discretisation of
  three-dimensional compressible Euler equations.
\end{itemize}

The rest of this paper is organised as follows. In
Section \ref{sec:Hilbert complexes} we introduce the concept of
Hilbert complex and establish the notation.
In Section \ref{sec:2D}, we describe compatible
finite elements in two dimensions and review some relevant analytical
results. In Section \ref{sec:linear fem} we review how these results
are applied to the linear rotating shallow water equations (including
a proper analysis of inertial modes). In Section \ref{eq:nonlinear sw}
we review how compatible finite element methods can be used to build
discretisations of the nonlinear rotating shallow water equations with
conservation properties. In Section \ref{sec:3d} we describe how to
extend compatible finite element spaces to three dimensions for the
columnar meshes used in numerical weather prediction (including new
results on approximation properties for compatible finite element
spaces on columnar meshes). In Section \ref{sec:temp} we describe how
to treat the discretisation of temperature within this framework
(including new results on existence and uniqueness of hydrostatic
balance).  Finally, we give a summary and outlook in Section
\ref{sec:sum}.

\section{Hilbert complexes}\label{sec:Hilbert complexes}
In this section we introduce the concept of Hilbert complex, which is the main ingredient to construct compatible finite element spaces, and we establish the notation for the Sobolev spaces used throughout the paper.

 
Let $\Omega$ be an $n$-dimensional domain with $n\leq3$, or in general a bounded Riemmannian manifold with or without boundary. Scalar and vector function spaces on
$\Omega$ can be collected in a diagram which has the following form:
\begin{equation}\label{eq:gcomplex}
\begin{tikzcd}[column sep=2em]
0 \arrow{r}&
V^0(\Omega) \arrow{r}{\ed^0}&
V^1(\Omega) \arrow{r}{\ed^1}&
\ldots \arrow{r}{\ed^{n-1}}&
V^n(\Omega) \arrow{r}& 0\, ,
\end{tikzcd}
\end{equation}
where $V^k(\Omega)$ is a scalar function space for $k=0,n$ and a vector function space for $0<k<n$. The operators $d^k$ for $k=0,\ldots,n-1$ are derivative operators, \emph{i.e.}\ they satisfy the Leibniz rule with respect to an opportunely defined product, and they verify $d^k \circ d^{k-1}=0$ for all $k=1,\ldots,n-1$. If the spaces $V^k(\Omega)$ are Hilbert spaces, than the sequence in \eqref{eq:gcomplex} is a \emph{Hilbert complex}.

If $n=1$,  $\Omega$ is locally isomorphic to an open interval in $\mathbb{R}$. If $x$ denotes a local coordinate on $\Omega$, we choose as Hilbert complex on $\Omega$ the sequence 
\begin{equation}\label{eq:1com}
\begin{tikzcd}[column sep=2em]
0\arrow{r}&
H^1(\Omega)\arrow{r}{\frac{\ed}{\ed x}} &
L^2(\Omega) \arrow{r}& 0\, .
\end{tikzcd}
\end{equation}
For $n=2$, we consider the Hilbert complex
\begin{equation}\label{eq:2com}
\begin{tikzcd}[column sep=2em]
0\arrow{r}&
H^1(\Omega)\arrow{r}{\nabla^\perp} &
H(\mr{div};\Omega)\arrow{r}{\nabla\cdot}  &
L^2(\Omega) \arrow{r}& 0\, .
\end{tikzcd}
\end{equation}
Finally, for $n=3$, we consider the Hilbert complex
\begin{equation}\label{eq:3com}
\begin{tikzcd}[column sep=2em]
0\arrow{r}&
H^1(\Omega)\arrow{r}{\nabla} &
H(\rm{curl};\Omega)\arrow{r}{\nabla\times}&
H(\rm{div}; \Omega) \arrow{r}{\nabla\cdot}&
L^2(\Omega) \arrow{r}&
0 \, .\end{tikzcd}
\end{equation}
From a numerical perspective reproducing such complex structures is key in achieving a number of properties which are crucial for geophysical fluid dynamics simulations, as it will be shown in the following sections.

It will be useful to have a unified notation to denote the spaces and the differential operators in the diagrams in \eqref{eq:1com}, \eqref{eq:2com} and \eqref{eq:3com}. For this purpose we introduce the following diagram,
\begin{equation}\label{eq:hcom}
\begin{tikzcd}[column sep=2em]
0\arrow{r}&
HV^0(\Omega)\arrow{r}{\ed^0} &
HV^1(\Omega)\arrow{r}{\ed^1}  &
\ldots\arrow{r}{\ed^{n-1}}  &
HV^n(\Omega) \arrow{r}& 0\, ,
\end{tikzcd}
\end{equation}
where the spaces $HV^k(\Omega)$ is the $k^{th}$ space in any of the sequences in \eqref{eq:1com}, \eqref{eq:2com} or \eqref{eq:3com}, depending on the dimension $n$ of $\Omega$. We denote by $\| \cdot \|_{L^2(\Omega)}$ (or simply $\|\cdot\|_{L^2}$ when the domain is clear from the context) the $L^2$ norm on any of these spaces, and we do not change notation depending on whether $V^k(\Omega)$ contains scalar or vector
functions. 

The spaces $HV^k(\Omega)$, are Hilbert spaces with respect to the norm  $\| \cdot \|_{HV^k(\Omega)}$, which is defined by
\begin{equation}
 \| u \|_{HV^k(\Omega)}^2 := \| u\|^2_{L^2(\Omega)} +\| \ed^k u\|^2_{L^2(\Omega)} \, ,
\end{equation}
for all $u\in HV^k(\Omega)$. For example, we have for $n=2,3$, 
\begin{equation}
 \| u \|_{HV^{n-1}(\Omega)}^2 =  \| u \|^2_{H(\mr{div};\Omega)}   := \| u\|^2_{L^2(\Omega)} +\| \nabla \cdot u\|^2_{L^2(\Omega)} \, .
\end{equation}
Furthermore, we define  $H^s V^k(\Omega)$ to be the spaces of scalar or vector functions with components in $H^s(\Omega)$. The norm associated to any of these spaces is defined by summing up the norms of each component in $H^s(\Omega)$. Analogous definitions hold for the spaces $W^s_pV^k (\Omega)$.

\section{Compatible finite element spaces in two dimensions}
\label{sec:2D}
In this section we provide a brief introduction to to compatible
finite element spaces in two dimensions and review some of their
properties. We start with a definition.
\begin{definition}
  Let $V_h^0(\Omega)\subset H^1(\Omega)$, $V_h^1(\Omega) \subset H(\ddiv;\Omega)$, $V_h^2(\Omega)
  \subset L^2(\Omega)$ be a sequence of finite element spaces ($V_h^0(\Omega)$
  and $V_h^2(\Omega)$ contain scalar valued functions, whilst $V_h^1(\Omega)$ contains
  vector valued functions). These spaces are called \emph{compatible}
  if:
  \begin{enumerate}
  \item $\nabla^\perp\psi \in V_h^1(\Omega)$, $\forall \psi \in V_h^0(\Omega)$,
  \item $\nabla\cdot u \in V_h^2(\Omega)$, $\forall u \in V_h^1(\Omega)$, and
  \item there exist bounded projections $\pi^i$, $i=0,1,2$,
    such that the following diagram commutes.
     \begin{equation}
     \begin{tikzcd}[column sep=2em]
       H^1(\Omega)\arrow{r}{\nabla^\perp}\arrow{d}{\pi^0} &
       H(\rm{div}; \Omega) \arrow{r}{\nabla\cdot}\arrow{d}{\pi^1}&
       L^2(\Omega) \arrow{d}{\pi^2} \\
       V_h^0(\Omega) \arrow{r}{\nabla^\perp}&
       V_h^1(\Omega) \arrow{r}{\nabla\cdot}&
       V_h^2(\Omega)
     \end{tikzcd}
     \end{equation}
  \end{enumerate}
\end{definition}

The use of these commutative diagrams to prove stability and
convergence of mixed finite element methods for elliptic problems has
a long history, as detailed in \cite{boffi2013mixed}. More recently,
this structure has been translated into the language of differential
forms. This provides a unifying framework that relates properties
between the various spaces and for different dimensions, called
\emph{finite element exterior calculus} \cite{Arnold06,Arnold10}. In
computational electromagnetism, the term \emph{discrete differential
  forms} is used to denote this choice of finite element spaces
\cite{hiptmair2002finite}.  Here, since we use the language of vector
calculus, we use the term \emph{compatible finite element spaces}.

If the domain $\Omega$ has a boundary $\partial\Omega$, then we define
$\mathring{H}^1(\Omega)$ and $\mathring{V}_h^0(\Omega)$ as the subsets of
$H^1(\Omega)$ and $V_h^0(\Omega)$ where functions vanish on
$\partial\Omega$, respectively.  Similarly, we define
$\mathring{H}(\ddiv;\Omega)$ and $\mathring{V}_h^1(\Omega)$ as the
subsets of $H(\ddiv;\Omega)$ and $V_h^1(\Omega)$ where the normal
components of functions vanish on $\partial\Omega$,
respectively. Then, the above commutative diagram still holds with
subspaces appropriately substituted. Since these are the required
coastal boundary conditions for streamfunction and velocity, we shall
assume that we use these subspaces throughout the rest of the paper,
but shall drop the $\mathring{\cdot}$ notation for brevity.

A large number of these sets of spaces can be found in the Periodic Table
of Finite Elements \cite{arnold2014periodic}. Examples include:
\begin{itemize}
\item $(CG_k,RT_{k-1},DG_{k-1})$ on triangular meshes for $k>0$, where
  $CG_k$ denotes the continuous Lagrange elements of degree $k$,
  $RT_k$ denotes the Raviart-Thomas elements of degree $k$, and $DG_k$
  denotes the discontinuous Lagrange elements of degree $k$. (This corresponds
  to the $\mathcal{P}_k^-\Lambda^r$ families in the finite element exterior
  calculus.)
\item $(CG_k,BDM_{k-1},DG_{k-2})$ on triangular meshes for $k>1$,
  where $BDM_k$ is the Brezzi-Fortin-Marini element of degree
  $k$. (This corresponds to the $\mathcal{P}_k\Lambda^r$ families in
  the finite element exterior calculus.)
\item $(CG_k,RT_{k-1},DG_{k-1})$ on quadrilateral meshes for $k>0$,
  where $RT_k$ now denotes the quadrilateral Raviart-Thomas
  elements. (This corresponds to the $\mathcal{Q}^-_k\Lambda^r$ spaces
  in the finite element exterior calculus \cite{Arnold14}.)
\end{itemize}

Whilst explicit constructions of the triangular mesh examples above
exist on affine meshes, the approach that is usually taken
computationally is to define these elements on a canonical reference
cell (triangle or square) and then to define elements on each physical
cell \emph{via} the pullback mappings listed in Table~\ref{tab:Fstar}. More precisely, if $F_T$ is a mapping from the
reference cell $\hat{T}$ to the physical cell $T$, then finite element
functions are related \emph{via}
\begin{itemize}
\item $\psi \in V_h^0(\hat{T}) \implies \psi\circ F_T^{-1} \in V_h^0(T)$,
\item $u \in V_h^1(\hat{T}) \implies Ju/\det(J) \circ F_T^{-1} \in V_h^1(T)$,
\item $\rho \in V_h^2(\hat{T}) \implies \rho/\det(J) \circ F_T^{-1} \in V_h^2(T)$,
\end{itemize}
where $J$ is the Jacobian matrix $DF_T$. More precisely, given a decomposition $\mc T_h$ of the domain $\Omega$ in elements $T$, with maximum element diameter $h$, the finite element space $V^k_h(\Omega)$ on $\mc T_h$ is defined by requiring $V^k_h(T) = F^{-1*}_T(V^k_h(\hat{T}))$ and $V^k_h(\Omega)\subset HV^k(\Omega)$.
The last requirement enforces the correct inter-element continuity accordingly to the degree $k$ in the complex.

The transformation for $V_h^1(\Omega)$ is referred to as the contravariant Piola transformation for $H(\ddiv)$
elements. Note that for $V_h^1(\Omega)$, scaling of edge basis functions is also
necessary to achieve the correct inter-element continuity; see
\cite{rognes2009efficient} for details of how to implement this
efficiently. This transformation can be extended to surfaces embedded
in three dimensions so that the $V_h^1(\Omega)$ vector fields remain tangential
to the surface; see \cite{rognes2013automating} for details of
efficient implementation, including examples on the sphere.

It is also useful to define dual operators to the $\nabla^\perp$ and
$\nabla\cdot$ operators through integration by parts; these allow us
to approximate the curl of functions in $V_h^1$ and the gradient of functions
in $V_h^2$.
\begin{definition}\label{def:nablat}
  Let $v \in V_h^1(\Omega)$, $\rho \in V_h^2(\Omega)$. Then define
  $\tilde{\nabla}^\perp\cdot v \in V_h^0(\Omega)$, and $\tilde{\nabla}\rho \in
  V_h^1(\Omega)$ \emph{via}
  \begin{align}
    \int_\Omega \gamma \tilde{\nabla}^\perp\cdot v \diff x & =
    -\int_\Omega \nabla^\perp\gamma\cdot v \diff x, \quad \forall \gamma \in V_h^0(\Omega), \\
    \int_\Omega w \cdot \tilde{\nabla}\rho \diff x & =
    -\int_\Omega \nabla\cdot w \rho \diff x, \quad \forall w \in V_h^1(\Omega).
  \end{align}
\end{definition}

A very important result about compatible finite element spaces is the existence of a Helmholtz decomposition for $V_h^1(\Omega)$.
\begin{theorem}[Helmholtz decomposition]\label{th:hdec}
  Let $\Omega$ be a suitably smooth (see \cite{Arnold10} for details)
  domain, and let $V_h^0(\Omega)$, $V_h^1(\Omega)$, $V_h^2(\Omega)$ be compatible finite element
  spaces on $\Omega$. Define the subspaces
  \begin{align}
  B_h &= \{u\in V^1_h(\Omega)\,:\, u=\nabla^\perp\psi, \, \psi \in V_h^0(\Omega)\}, \\
  B^*_h &= \{u\in V^1_h(\Omega)\,:\, u=\tilde{\nabla}\phi, \, \phi \in V_h^2(\Omega)\}, \\
  \mathfrak{h}_h & = \{ u\in V^1_h(\Omega)\,:\, \nabla\cdot{u}=0, \, \tilde{\nabla}^\perp\cdot
  u=0\}.
  \end{align}
  Then,
  \begin{equation}
  V_h^1(\Omega) = B_h \oplus \mathfrak{h}_h \oplus B^*_h,
  \end{equation}
  where $\oplus$ indicates an orthogonal decomposition with respect to
  the $L^2$ inner product.\\ Further, the space $\mathfrak{h}_h$ of
  harmonic functions of $V_h^1(\Omega)$ has the same dimension as the
  corresponding space $\mathfrak{h}$ of harmonic functions of
  $H(\ddiv;\Omega)$ (which is determined purely from the topology of  $\Omega$).
\end{theorem}
\begin{proof}
  See \cite{Arnold06}.
\end{proof}

\subsection{Mixed discretisation for the Poisson equation}
The mixed discretisation of the Poisson equation
\begin{equation}
  -\nabla^2 p = f,
  \quad \int_\Omega p \diff x = 0, \quad \pp{p}{n}=0 \mbox{ on }
  \partial \Omega,
\end{equation}
using compatible finite elements is very well studied. Defining
\begin{equation}
  \bar{V}_h^2(\Omega) = \left\{p\in V_h^2: \int_\Omega p \diff x = 0\right\},
\end{equation}
we seek
$u^h\in V_h^1(\Omega)$ and $p^h \in \bar{V}_h^2(\Omega)$ such that
\begin{align}
  \label{eq:poisson u}
  \int_\Omega w^h \cdot u^h \diff x + \int_\Omega (\nabla\cdot w^h) p^h \diff x & = 0,
  \quad \forall w^h \in V_h^1(\Omega), \\   \label{eq:poisson p}
  \int_\Omega \phi^h\nabla\cdot u^h \diff x & = \int_\Omega \phi^h f \diff x, \quad
  \forall \phi^h \in \bar{V}_h^2(\Omega).
\end{align}
The compatible finite element structure between $V_h^1(\Omega)$ and $V_h^2(\Omega)$ is
behind the proof of the following results, all of which can
be found in \cite{boffi2013mixed}. The first result is about the
stability of the $\tilde{\nabla}$ operator defined in Definition~\ref{def:nablat}.
\begin{theorem}[inf-sup condition]
  Define the bilinear form
  \begin{equation}
    b(w,p) = \int_\Omega (\nabla\cdot w) p \diff x.
  \end{equation}
  There exists $c>0$, independent of mesh resolution, such that
  \begin{equation}
    \inf_{p\in \bar{V}^2_h(\Omega)}\sup_{w\in V_h^1(\Omega)}\frac{b(w,p)}{\|w\|_{H(\ddiv)}
      \|p\|_{L^2}} \geq c,
  \end{equation}
  where
  \begin{equation}
    \|w\|^2_{H(\ddiv)} = \|w\|^2_{L^2}+\|\nabla\cdot w\|^2_{L^2}.
    \end{equation}
\end{theorem}
\begin{proof}
The result is a  particular case of Proposition 5.4.2 in \cite{boffi2013mixed}.
\end{proof}
This result is then key to the proof of the following theorem.
\begin{theorem}
  Equations (\ref{eq:poisson u}-\ref{eq:poisson p}) have a unique
  solution ($u^h$,$p^h$), satisfying
  \begin{align}
    \|u^h\|_{H(\ddiv)} & \leq a_1 \|f\|_{L^2}, \\
    \|p^h\|_{L^2} & \leq a_2 \|f\|_{L^2}, \\
    \|u-u^h\|_{H(\ddiv)} + \|p-p^h\|_{L^2} & \leq a_3 \|f-f^h\|_{L^2},
  \end{align}
  for positive constants $a_1$, $a_2$, $a_3$ independent of mesh resolution,
  where $p$ is the exact solution to the Poisson equation, $u=-\nabla p$,
  and $f^h$ is the $L^2$ projection of $f$ into $V_h^2(\Omega)$.
\end{theorem}
\begin{proof}
This is a particular case of the abstract results in Chapters 4 and 5 of \cite{boffi2013mixed}. See in particular Theorems 4.3.2, 5.2.1 and 5.2.5.
\end{proof}
Finally, we have the following convergence result for eigenvalues.
\begin{theorem}
  Let $\lambda_0<\lambda_1<\ldots<\lambda_i<\ldots$ be the eigenvalues
  of the Laplacian operator $-\nabla^2$, with corresponding
  eigenspaces $\{E_i\}_{i=0}^{\infty}$. \\ Let
  $\lambda^h_0<\lambda_1^h<\ldots<\lambda_m^h$, be the complete set of
  eigenvalues of the discrete Laplacian operator
  $-\nabla\cdot\tilde{\nabla}$, with corresponding eigenspaces
  $\{E_i^h\}_{i=0}^{m}$.

  Then $\forall \epsilon > 0$, $\forall k
  \in \mathbb{N}$, $\exists h_0 > 0$ such that $\forall h \leq h_0$,
  \begin{align}
    \max_{i=0,\ldots, k}|\lambda_i - \lambda^h_i| &\leq \epsilon, \\
    \hat{\delta}\left(
    \oplus_{i=0}^kE_i,     \oplus_{i=0}^kE_i^h 
    \right) & \leq \epsilon,
  \end{align}
  where
  \begin{align}
    \hat{\delta}(E,F) & = \max\left(\delta(E,F),\delta(F,E)\right), \\
    \delta(E,F) & = \sup_{p\in E,\|p\|_{L^2}=1}\inf_{\phi\in F}\|p-\phi\|_{L^2}.
  \end{align}
\end{theorem}
\begin{proof}
The result follows from applying the abstract results of Section 6.5 in \cite{boffi2013mixed}, and in particular as a consequence of Theorem 6.5.1.
\end{proof}

\section{Compatible finite element methods for the linear shallow water equations}
\label{sec:linear fem}
In this section we summarise some properties of compatible finite
element methods applied to the linear shallow water equations. The
equations take the form
\begin{align}
  \label{eq:cont sw u}
  u_t + fu^\perp + g\nabla h & = 0, \\
  \label{eq:cont sw h}
  h_t + H\nabla\cdot u & = 0,
\end{align}
where the $\perp$ operator maps $(u_1,u_2)$ to $(-u_2,u_1)$, $u_t:=
\partial u/\partial t$, $h$ is the layer height, and $f$ is the
(possibly spatially-dependent) Coriolis parameter, $g$ is the
acceleration due to gravity and $H$ is a constant representing the reference depth. 
This system of equations is a toy model of barotropic waves in the atmosphere, and it is used to explore important properties of numerical discretisations.

The compatible finite element spatial discretisation
for the shallow water equations takes $u\in V_h^1(\Omega)$, $h\in V_h^2(\Omega)$,
and solves
\begin{align}  
  \label{eq:sw u}
  \int_\Omega w \cdot u_t \diff x + \int_\Omega f w\cdot u^\perp \diff x
  - g\int_\Omega (\nabla\cdot w) h \diff x & = 0, \quad \forall w \in V_h^1(\Omega),   \\
  \label{eq:sw h}
  \int_\Omega \phi (h_t + H\nabla\cdot u)\diff x & = 0, \quad
  \forall \phi \in V_h^2(\Omega). 
\end{align}
\cite{StTh2012} provided a number of desirable properties for
numerical discretisations when applied to the linear shallow water
equations, which we consider in the following subsections. We note
that all of these results apply on arbitrarily unstructured grids.

\subsection{Energy conservation}
Combining Equation~\eqref{eq:sw u} and \eqref{eq:sw h}, with $w=u$ in the first and $\phi = h$ in the second, shows that the energy
\begin{equation}
  E = \frac{1}{2}\int_\Omega H|u|^2 + gh^2 \diff x,
\end{equation}
is conserved by these solutions. Note that in the fully discrete system, \emph{i.e.}\ when time is also discretised, the level of energy conservation will depend upon the chosen scheme and timestep.

\subsection{Stability of pressure gradient operator}
It is crucial that the discretisation of the pressure gradient term
$\nabla h$ is stable; this means that there are no functions $h$ for
which $h$ is large amplitude but the discrete approximation to $\nabla
h$ is small. In other words, there are no spurious eigenvalues for the
discrete Laplacian obtained by combining $\nabla\cdot$ with
$\tilde{\nabla}$.  For mixed finite element discretisations, this
requirement is formalised by the inf-sup condition, which we have
already stated in the previous section. In fact, further than this,
the eigenvalue convergence result for $\nabla\cdot\tilde{\nabla}$
shows that the wave equation with $f=0$ has convergent eigenvalues.

\subsection{Geostrophic balance}
A crucial property for large scale modelling is that the linearised equations on the $f$-plane (\emph{i.e.}\ $\Omega$ is the periodic plane and $f$ is constant) have geostrophic steady states \cite{StTh2012}. This is proved in the following.
\begin{theorem}[Cotter and Shipton (2012)]
  For all divergence-free $u\in B_h$, there exists $h$ such that
  $(u,h)$ is a steady state solution of Equations (\ref{eq:sw
    u}-\ref{eq:sw h}).
\end{theorem}
\begin{proof}
  Since $\nabla\cdot u=0$, Equation \eqref{eq:sw h} immediately
  implies that $h_t=0$. Since, $u\in B_h$ and so there exists $\psi\in
  V_h^0(\Omega)$ such that $u=\nabla^\perp\psi$.  Take $h\in
  V_h^2(\Omega)$ according to
  \begin{equation}
    \int_\Omega \phi g h \diff x = 
    \int_\Omega \phi f \psi\diff x 
    \quad \forall \phi \in V_h^2(\Omega).
  \end{equation}
  Then,
  \begin{align}
    \int_\Omega w \cdot u_t \diff x & = \int_\Omega f w\cdot \nabla \psi
    +(\nabla\cdot w) gh \diff x, \\
     & = \int_\Omega -(\nabla\cdot w) f\psi
    +(\nabla\cdot w) gh \diff x, \\
         & = \int_\Omega -(\nabla\cdot w) g h
    +(\nabla\cdot w) gh \diff x=0,
  \end{align}
  after integrating by parts, which is permitted since the product $ \psi w \in H(\ddiv)$.
\end{proof}
If this property is not satisfied, then a solution with balanced initial condition oscillates on the fast timescale. This is catastrophic when the nonlinear terms are introduced since the solution is polluted by rapid oscillations before the slow balanced state has had time to evolve; this renders a discretisation useless for large scale atmosphere or ocean modelling. 

\subsection{Absence of spurious inertial oscillations}
Inertial oscillations are solutions of the shallow water equations on
the $f$-plane with $h=0$. For the continuous
equations, Equation \eqref{eq:cont sw h} implies that $\nabla\cdot u =
0$. Then, Equation \eqref{eq:cont sw u} implies that
\begin{equation}
  u_t + f u^\perp = 0.
\end{equation}
Taking the divergence then implies that $\nabla^\perp \cdot u =
0$. Therefore, $u$ is a harmonic function. On the periodic plane,
$\mathfrak{h}$ is two-dimensional, consisting of all spatially
constant velocity fields, and the solution just rotates around
with frequency $f$.

Some discretisations, for example the $DG_k$-$CG_{k-1}$ discretisation considered in \cite{cotter2011numerical}, have subspaces of spurious inertial oscillations, \emph{i.e.}\ extra divergence-free oscillatory solutions with $h=0$, which are not spatially uniform. These are known to cause spurious gridscale oscillations in regions of high shear such as representations of the Gulf Stream in ocean models \cite{danilov}. The staggered C-grid finite
difference method is free from such modes. It does, however, have a steady divergence-free solution in the kernel of the Coriolis operator, which
is known as a ``Coriolis mode''. We shall see in this section that similar
modes can exist for compatible finite element methods.

Inertial oscillations for compatible finite element discretisations
were not properly considered by \cite{CoSh2012}, so we consider them
here.  We start by observing that for a compatible discretisation $\mathfrak{h}_h \equiv \mathfrak{h}$, \emph{i.e.}\ it is the space of constant vector fields $v_0$. To see this, first note that $\nabla\cdot v_0=0$.
Secondly,
\begin{equation}
  \int_\Omega \nabla^\perp\gamma \cdot v_0 \diff x =
  - \int_\Omega \gamma\underbrace{\nabla^\perp v_0}_{=0}\diff x = 0,
  \quad \forall \gamma\in V_h^0(\Omega),
\end{equation}
after integration by parts (since $\gamma$ and $v_0$ are both
continuous), and therefore $\tilde{\nabla}^\perp v_0=0$. Therefore,
$v_0\in\mathfrak{h}_h$. Since $\dim(\mathfrak{h}_h)= \dim(\mathfrak{h})=2$ by Theorem~\ref{th:hdec}, we have
therefore characterised all discrete harmonic functions in
$V_h^1(\Omega)$. 

\begin{theorem}
  The only time-varying solutions of the compatible finite element
  discretisation of the shallow water equations on the $f$-plane with
  divergence-free $u$ and $h=0$ are the solutions with spatially
  constant $u$ satisfying $u_t + fu^\perp = 0$. Any time-independent
  solutions are in the kernel of the discrete Coriolis operator,
  \emph{i.e.}
  \begin{equation}
    \int_\Omega w\cdot u^\perp \diff x = 0, \quad \forall w \in V_h^1(\Omega).
  \end{equation}
\end{theorem}
\begin{proof}
To see that spurious inertial oscillations do not exist for
compatible finite element discretisations, we first note that $h=0$
implies that $\nabla\cdot u = 0$. Then Equation \eqref{eq:sw u}
implies that
\begin{equation}\label{eq:simprsw}
  \int_\Omega w \cdot u_t \diff x + f\int_\Omega w\cdot u^\perp \diff x  = 0,
  \quad \forall w \in V_h^1(\Omega).
\end{equation}
Taking $w=\nabla^\perp\gamma$ for $\gamma\in V_h^0(\Omega)$, we obtain
\begin{align}
  \int_\Omega \nabla^\perp\gamma \cdot u_t \diff x
  &= -f\int_\Omega \nabla\gamma\cdot u \diff x, \\
  & = f\int_\Omega \gamma \nabla\cdot{u}\diff x = 0,
  \quad \forall \gamma \in V_h^0(\Omega), 
\end{align}
where we may integrate by parts for the same reason as in the
geostrophic balance calculation. This means that $u_t$ is orthogonal
to everything in $B_h$, \emph{i.e.}\ $u_t\in \mathfrak{h}_h$.
Hence, we
may write $u=k+u_0$, where $k\in \mathfrak{h}_h$, and $u_0$ is
divergence-free and independent of time. Therefore, by the geostrophic
balance condition, there exists $h_0\in V^2_h(\Omega)$ such that
\begin{equation}
f\int_\Omega w \cdot u_0^\perp \diff x -
  g\int_\Omega (\nabla\cdot w) h_0 \diff x = 0,
  \quad \forall w \in V_h^1(\Omega),
\end{equation}
and therefore substitution of $u=k + u_0$ into Equation~\eqref{eq:simprsw} leads to
\begin{equation}
  \int_\Omega w\cdot k_t \diff x
  = -f\int_\Omega w \cdot k^\perp \diff x +
  g\int_\Omega (\nabla\cdot w) h_0 \diff x,
  \quad \forall w \in V_h^1(\Omega).
\end{equation}
Since $\mathfrak{h}_h$ is the space of constant vector
fields, if $k\in \mathfrak{h}_h$, then so is $k^\perp$. Therefore both
$k_t$ and $k^\perp$ are orthogonal to $B_h^*$, and therefore
\begin{equation}
  \int_\Omega (\nabla\cdot w) h_0 \diff x = 0, \quad \forall w \in B_h^*.
\end{equation}
This means that if we define $v\in V^1_h(\Omega)$ \emph{via} 
\begin{equation}
 \int_\Omega w\cdot v \diff x = -\int_\Omega (\nabla\cdot w) h_0 \diff x,
\quad \forall w \in V^1_h(\Omega),
\end{equation}
then $\nabla \cdot v=0$. This means that $v \in V_h^1(\Omega)$ and $h_0 \in V^2_h(\Omega)$ satisfy the following coupled system of equations,
\begin{align}
  \int_\Omega w\cdot v \diff x + \int_\Omega (\nabla\cdot w) gh_0 \diff x & = 0,
  \quad \forall w \in V_h^1(\Omega), \\
  -\int_\Omega \phi\nabla\cdot v \diff x & = 0, \quad \forall \phi\in V_h^2(\Omega).
\end{align}
This is just the mixed finite element discretisation of the Poisson equation
with zero on the right-hand side, and therefore $h_0=0$ by
uniqueness. Therefore, we have 
\begin{equation}
\int_\Omega w \cdot u^\perp_0 \diff x = 0, \quad \forall w \in V_h^1(\Omega).
\end{equation}
This means that $u_0$ must be in the kernel of the
Coriolis operator.

Finally, we have
\begin{equation}
  \int_\Omega w\cdot k_t \diff x + \int_\Omega fw\cdot k^\perp \diff x = 0,
  \quad \forall w\in \mathfrak{h}_h\subset V_h^1(\Omega),
\end{equation}
which implies the pointwise equation $k_t + f k^\perp=0$, since the $L^2$-projection onto $\mathfrak{h}_h$ is trivial. 
\end{proof}

\subsection{Inertia-gravity waves}
Continuing the analysis of solutions in the periodic $f$-plane,
\cite{CoSh2012} considered the time-dependent divergent solutions of Equations
(\ref{eq:sw u}-\ref{eq:sw h}), \emph{i.e.}\ the inertia-gravity wave
solutions. One way to calculate the equation for these modes is to
make a discrete Helmholtz decomposition of both the solution
$u$ and the test function $w$,
\begin{align}
  u &= \nabla^\perp\psi + \tilde{\nabla}\phi + k, \quad \psi\in\bar{V}_h^0(\Omega), \,
  \phi\in \bar{V}_h^2(\Omega), \, k \in \mathfrak{h}_h, \\
  w &= \nabla^\perp\gamma + \tilde{\nabla}\alpha + l,
  \quad \gamma\in \bar{V}_h^0(\Omega), \,
  \alpha\in \bar{V}_h^2(\Omega), \, l \in \mathfrak{h}_h,
\end{align}
where
\begin{equation}
  \bar{V}_h^0(\Omega) = \left\{\psi\in V_h^0: \int_\Omega \psi \diff x = 0\right\},
\end{equation}
and substitute into the equations, using orthogonality to obtain
\begin{align}
  \int_\Omega \nabla\gamma\cdot\nabla \psi_t - f\nabla\gamma\cdot\tilde{\nabla}\phi\diff x & = 0, \quad \forall \gamma \in \bar{V}_h^0(\Omega), \label{eq:ortpro}\\
  \int_\Omega \tilde{\nabla}\alpha\cdot\tilde{\nabla}{\phi}_t +
  f\tilde{\nabla}\alpha\cdot(\tilde{\nabla}\phi)^\perp - \tilde{\nabla}\alpha
  \cdot\nabla\psi - g(\nabla\cdot\tilde{\alpha} )h \diff x & = 0,
  \quad \forall \alpha\in \bar{V}_h^2(\Omega),
\end{align}
where we have made use of the fact that the discrete harmonic functions
consists of constant vector fields.
Taking the time derivative of the second equation and substituting
$h_t=-H\nabla\cdot u$, which holds pointwise since $\nabla\cdot u\in
V_h^2(\Omega)$, we obtain
\begin{equation}
  \int_\Omega \tilde{\nabla}\alpha\cdot\tilde{\nabla}{\phi}_{tt} +
  f\tilde{\nabla}\alpha\cdot(\tilde{\nabla}\phi)^\perp_t -
  \tilde{\nabla}\alpha \cdot\nabla\psi_t +
  gH\nabla\cdot\tilde{\nabla}{\alpha}\nabla\cdot\tilde{\nabla}{\phi} \diff x = 0,
  \quad \forall \alpha\in \bar{V}_h^2(\Omega).
\end{equation}
Finally, we define a projection $P:\bar{V}_h^2(\Omega)\to \bar{V}_h^0(\Omega)$ by
\begin{equation}
  \int_\Omega \nabla\gamma\cdot\nabla P\phi = \int_\Omega\nabla\gamma\cdot\tilde{\nabla}\phi\diff x, \quad \forall \gamma \in \bar{V}^0_h(\Omega), \\
\end{equation}
which is solvable since the left hand side is the standard continuous finite
element discretisation of the Laplace operator. Then, the equation becomes
\begin{equation}
  \int_\Omega \tilde{\nabla}\alpha\cdot \tilde{\nabla}\phi_{tt}
  + f^2\nabla P\alpha \cdot \nabla P\phi +
    gH\nabla\cdot\tilde{\nabla}{\alpha}\nabla\cdot\tilde{\nabla}{\phi} \diff x = - \int_{\Omega} f\tilde{\nabla}\alpha\cdot(\tilde{\nabla}\phi)^\perp_t \diff x,
  \quad \forall \alpha\in \bar{V}_h^2(\Omega).
\end{equation}
Upon inspection, when $f=0$ this is just the mixed compatible finite
element discretisation of the wave equation, which we know to have
convergent eigenvalues from the previous section.

The continuous form of this equation is
\begin{equation}
  \nabla^2(\pp{^2}{tt} + f^2)\phi + gH\nabla^2\nabla^2\phi = 0.
\end{equation}
For $f\neq 0$, we might hope that the right-hand side of the
discretised equation is small for resolved modes.  However, the
projection is more problematic, since $V_h^0(\Omega)$ and $V_h^2(\Omega)$ might have
different dimensions. If $\dim(V_h^0(\Omega))<\dim(V_h^2(\Omega))$, then some modes are
projected out and the $f^2$ contribution (which originates from the
Coriolis term) is invisible to those modes. Hence, \cite{CoSh2012}
concluded that $\dim(V_h^0(\Omega))\geq \dim(V_h^2(\Omega))$ is a necessary (but not
sufficient) condition for absence of spurious inertia-gravity
modes. This necessary condition is not satisfied for the $RT_k$ family
on triangles. For example, when $k=0$, we have $\dim(V_h^2(\Omega))=2\dim(V_h^0(\Omega))$.
This leads to the spurious inertia-gravity mode behaviour that was
observed by \cite{Da2010} for $RT_0-DG_0$ and the related C-grid
staggered finite difference method (obtained by lumping the mass in
the $RT_0-DG_0$ discretisation). On the other hand, the $BDM_k$ family
does satisfy this necessary condition. \cite{CoSh2012} noted that the
$BDFM_1$ spaces exactly satisfy $\dim(V_h^2(\Omega))=\dim(V_h^0(\Omega))$ which might be
additionally desirable, particularly since they also analysed the
Rossby wave equation (not reproduced here) and concluded that
$\dim(V_h^2(\Omega)) \leq \dim(V_h^0(\Omega))$ is a necessary condition for absence of
spurious Rossby waves. However, it is not clear that spurious Rossby
waves are damaging for numerical simulations, since high spatial
frequency Rossby waves do not propagate anyway.

To explore sufficient conditions for absence of spurious
inertia-gravity modes, it might only be possible to make progress by
assuming symmetric grids and performing von Neumann analysis. On
triangular grids this is very challenging, but the analysis has been
successfully carried out for $RT_0$ and $BDM_1$ spaces in
\cite{rostand2008raviart}; it is still unclear how to dealias the high
mode branch identified for $BDM_1$, but we believe that they are physical
modes. On rectangular elements, it is much easier since one can
exploit the tensor product structure. \cite{staniforth2013analysis}
showed that (under the assumption of $y$-independent solutions) the
$RT_0-DG_0$ space is free from spurious modes and has a similar, but
more accurate dispersion relation to the finite difference C-grid
staggering. They investigated $RT_1-DG_1$, and found two mode branches,
a high branch and a low branch. Although both of these branches
correspond to physical modes, there is a small jump between the
branches, where the group velocity goes to zero. This might be
problematic in a nonlinear model if numerical noise, data assimilation
or physics parameterisations were to excite this mode, since it would
not propagate away. The authors were able to remove the jump, and the
zero group velocity region, by applying a perturbation to the mass
matrix (which does not affect consistency). \cite{melvin2014two}
considered two-dimensional solutions and showed the same results; no
spurious inertia gravity modes but jumps do appear in the dispersion
relation that can be fixed by consistent modification of the mass
matrix.

\section{Compatible finite element methods for the nonlinear shallow water
  equations}
\label{eq:nonlinear sw}
In this section we survey the application of compatible finite element
methods to the nonlinear shallow water equations.

\subsection{Energy-enstrophy conservation}
The compatible
finite element structure facilitates the design of spatial discretisations 
that conserve mass, energy and enstrophy. An energy-enstrophy conserving
scheme was presented in \cite{McRae14}. The starting point is the
vector invariant form of the shallow water equations,
\begin{align}
  u_t + qF^\perp + \nabla \left(gD + \frac{1}{2}|u|^2\right) & = 0, \\
  D_t + \nabla\cdot F & = 0, 
\end{align}
where $D$ is the layer depth, $F$ is the mass flux and $q$
is the potential vorticity, defined by
\begin{equation}
  F = uD, \quad q = \frac{\nabla^\perp\cdot u + f}{D},
\end{equation}
respectively. It is well known that these equations have conserved energy $E$
and mass $M$, where
\begin{equation}
  E = \frac{1}{2}\int_\Omega D|u|^2 + gD^2\diff x, \quad
  M = \int_\Omega D \diff x,
\end{equation}
plus an infinity hierarchy of conserved Casimirs, given by
\begin{equation}
  C_i = \int_\Omega Dq^i \diff x, \quad i=1,2,\ldots.
\end{equation}
The energy-enstrophy conserving scheme, which can be thought of as an
extension of the scheme of \cite{arakawa1981potential} to compatible
finite element methods, conserves $E$, $M$, the total vorticity $C_1$
and the enstrophy $C_2$. The scheme takes $q\in V_h^0(\Omega)$, $u,F\in V_h^1(\Omega)$,
and $D \in V_h^2(\Omega)$, with
\begin{align}
  \label{eq: nonlinear u}
  \int_\Omega w \cdot u_t \diff x + \int_\Omega w\cdot F^\perp q\diff x
  &= \int_\Omega \nabla\cdot w \left(gD + \frac{1}{2}|u|^2\right) \diff x, \quad \forall w \in V_h^1(\Omega), \\
  \int_\Omega \phi \left(D_t + \nabla\cdot F\right)\diff x & = 0, \quad
  \forall \phi \in V_h^2(\Omega), \\
  \int_\Omega w\cdot F\diff x & = \int_\Omega w\cdot u D \diff x, \quad
  \forall w \in V_h^1(\Omega), \\
  \int_\Omega \gamma qD \diff x &= -\int_\Omega \nabla^\perp \gamma \cdot u
  + \gamma f \diff x, \quad \forall \gamma \in V_h^0(\Omega).
  \label{eq: nonlinear q}
\end{align}
In this scheme $u$ and $D$ are the prognostic variables and $F$ and
$q$ can be diagnosed from them using their definitions in integral
form. Despite $u$ and $D$ being prognostic, one can derive a potential
vorticity equation by taking $w=\nabla^\perp\gamma$ in Equation
\eqref{eq: nonlinear u}, for $\gamma\in V_h^0(\Omega)$, which is permitted due
to the compatible structure. Then, combining this with the time
derivative of Equation \eqref{eq: nonlinear q}, we obtain
\begin{equation}
  \label{eq: nonlinear q conservation}
  \int_\Omega \gamma (qD)_t \diff x -\int_\Omega \nabla\gamma \cdot Fq\diff x = 0,
  \quad \forall \gamma \in V_h^0(\Omega),
\end{equation}
which is the standard finite element discretisation of the potential
vorticity equation. Direct computation plus integration by parts then
shows that this scheme conserves $M$, $E$, $C_0$ and $C_1$. Enstrophy
conservation is particularly important because it specifies control of
$q$ in the $L^2$ norm (weighted by $D$), which implies control of
gradients of the divergence-free component of $u$ \emph{via} the
discrete Helmholtz decomposition. Without this control, even with
energy conservation implying control of $u$ in $L^2$, we would observe
unbounded growth in the derivatives of $u$ as the mesh is refined. 

In two dimensions, although energy cascades to large scales, enstrophy
cascades to small scales. Since Equation \eqref{eq: nonlinear q
  conservation} is a regular Galerkin discretisation of the law of
conservation of potential vorticity, this means that enstrophy will
accumulate at the grid scale in general, leading to oscillations.  It
was noted by \cite{McRae14} that it is possible to modify the
potential vorticity flux $qF$ in such a way that enstrophy is
dissipated at the gridscale whilst energy is still conserved, by introducing a finite element version of the anticipated potential vorticity scheme \cite{Ringler10}. This was demonstrated to produce less oscillatory solutions in vortex merger experiments.

All of the properties described in this section can be expressed in
the language of finite element exterior calculus; this was presented in
\cite{Cotter14}. Moreover, note that conservation of energy and enstrophy may not be exact when time is also discretised. Then, the level of conservation will depend upon the chosen scheme and timestep.

\subsection{Stable transport schemes}

When these schemes are extended to three dimensions, we encounter the
problem that energy is now also cascading to small
scales. Historically, the solution to this was to add eddy viscosity
terms to collect energy at small scales. However, this is
unnecessarily dissipative and so does not make the best use of
available resolution.  The modern approach to developing numerical
weather prediction models is to use high-order upwind transport
schemes to control energy and enstrophy in the numerical solution.
The shallow water equations provide a useful vehicle for exploring
these issues in a two-dimensional model that can be run quickly; we
review recent results here.

Since we are using different finite element spaces for different
fields, this calls for different approaches to the design of transport
schemes for them. For $D\in V_h^2(\Omega)$, a standard upwind Discontinuous Galerkin
discretisation can be applied to the mass conservation equation,
\begin{equation}
  \int_\Omega \phi D_t \diff x - \int_\Omega \nabla_h \phi\cdot uD \diff x
   + \int_\Gamma \jump{u\phi}\tilde{D}\diff x = 0, \quad \forall \phi \in V_h^2(\Omega),
\end{equation}
where $\Gamma$  is the set of interior facets in the mesh, the ``jump''
operator is defined for vectors $w$ by
\begin{equation}
  \jump{w} = w^+\cdot n^+ + w^-\cdot n^-,
\end{equation}
with $\pm$ indicating the values of functions on different sides of each
edge and $n^{\pm}$ is the unit normal vector pointing from the $\pm$ side
to the $\mp$ side, $\tilde{D}$ indicates the value of $D$ taken from the
upwind side, \emph{i.e.}\ the side where $u^\pm\cdot n^\pm$ is negative, 
and where $\nabla_h$ is the usual ``broken'' gradient defined in $L^2$ as
\begin{equation}
  \nabla_h|_T \phi = \nabla|_T \phi.
\end{equation}
Here $|_T$ indicates the restriction of a function to a single cell
$T$, for all cells $T$ in the mesh decomposition of $\Omega$.  It was
shown in \cite{Cotter14, shipton:_higher} that it is possible to find
a mass flux $F$ such that
\begin{equation}
  D_t + \nabla\cdot F = 0,
\end{equation}
in a (cheap) post-processing step after the transport scheme step
has been completed.

There are several different approaches to velocity transport. One
possibility is to maintain Equation \eqref{eq: nonlinear q}, and to
modify the potential vorticity flux $qF$ so that Equation \eqref{eq:
  nonlinear q conservation} is replaced by an upwind discretisation of
the potential vorticity conservation law. For example, choosing the
Streamline Upwind Petrov Galerkin \cite{brooks1982streamline}
discretisation
\begin{equation}
  \int_\Omega \left(\gamma + \frac{\eta h F}{|F|}\cdot \nabla \gamma\right)
  (Dq)_t \diff x  - \int_\Omega \nabla\gamma\cdot qF \diff x
  + \int_\Omega \frac{\eta h F}{|F|}\cdot\nabla\gamma \nabla \cdot (Fq)
  \diff x = 0, \quad \forall \gamma \in V_h^0(\Omega),
\end{equation}
where $0<\eta<1$ is a constant upwind parameter and $h(x)$ is an indicator
function returning the size of the cell containing the point $x$.
This is a high-order consistent upwind discretisation. We obtain
a modified potential vorticity flux
\begin{equation}
  qF - \frac{\eta h F}{|F|}\left((qD)_t + (Fq)\right),
\end{equation}
which can then replace $qF$ in Equation \eqref{eq: nonlinear u}.
Since  this flux is still proportional to $F$, energy is still
conserved (neglecting any energy changes from the $D$ transport
scheme). In \cite{shipton:_higher}, this idea was extended to higher
order Taylor-Galerkin timestepping schemes. A related approach is to
introduce dual grids, so that the dual grid potential vorticity is
represented in a discontinuous finite element space. At lowest order,
these are $P_0$ spaces (piecewise constant functions) that can
be treated with finite volume advection schemes that achieve higher
spatial order by increasing the stencil size. A discretisation for the
shallow water equations using finite element spaces for arbitrary
polygonal grids, and mappings between primal and dual grids, was
developed in \cite{thuburn2015primal}.   Results demonstrating this
approach on a standard set of test cases on the sphere and comparing
cubed sphere and dual icosahedral grids were also presented.

A difficulty with the approach described above is the loss of
consistency near the boundary in Equation \eqref{eq: nonlinear u}, due
to the fact $V_h^0(\Omega)$ consists of functions that vanish on the boundary.
It can be shown that this results in a loss of consistency of one
order, \emph{i.e.}\ from 2nd order to 1st order consistency. There is a
similar problem with the dual grid approach. An alternative approach
to this is to treat the curl operator directly in the velocity
equation (still in vector invariant form) without introducing an
intermediate potential vorticity variable. This can be developed by
dotting with a test function $w$ and integrating over a single cell
$T$, then integrating by parts to obtain
\begin{equation}
  \int_T w\cdot \frac{F^\perp}{D}(\nabla^\perp\cdot u + f)\diff x =
  -\int_T \nabla^\perp\left(w\cdot \frac{F^\perp}{D}\right)
  \cdot u \diff x
  +\int_{\partial T} n^\perp\cdot \frac{F^\perp}{D}\tilde{u}\diff S,
\end{equation}
where $n$ is the unit outward pointing normal to $\partial T$, the
boundary of $e$, and $\tilde{u}$ is either the averaged or the upwind
value of $u$ on $\partial T$. Even in the presence of upwinding, this
discretisation is still energy conserving, since the energy equation
sets $w=F$, and $F\cdot F^\perp=0$. In
\cite{natale:_variat_h_finit_elemen_discr} it was shown that the
application of this discretisation can be derived from a variational
principle; the paper also proved convergence (suboptimal by 1 degree)
of solutions for the centred flux variant of the discretisation in two
and three dimensions.

Finally, we remark that all of these velocity transport schemes
require solution of global mass matrix systems, one for each stage of
an explicit Runge-Kutta method, for example. If one is willing to give
up energy conservation, then an Embedded Discontinuous Galerkin
transport scheme \cite{cotter2016embedded} can be used. In these
schemes, the advected field $u \in V_h^1(\Omega)$ is injected into the broken
space $\hat{V}_h^1(\Omega)$ (the same finite element space but with no
continuity constraints between cells) at the start of the
timestep. Then, one or more explicit upwind Discontinuous Galerkin
timesteps are taken using a stable Runge-Kutta scheme. Finally, the
solution is projected back into $V_h^1(\Omega)$. This final step does involve a
global mass solve but this can be incorporated into a semi-implicit
update at no extra cost. We are currently investigating the stability
and accuracy of these schemes.

\subsection{Semi-implicit implementation}
\label{sec:semi-implicit}

In semi-implicit discretisations, a crucial aspect is the efficient
solution of the resulting coupled linear system that must be solved
during each nonlinear iteration of the semi-implicit scheme. For
example, consider the following implicit timestepping scheme.
\begin{align} \nonumber
  \int_\Omega w\cdot u^d \diff x & =
  \int_\Omega w\cdot u^n \diff x +
  \frac{\Delta t}{2}\int_\Omega w \cdot f(u^n)^{\perp}\diff x  \\
&  \qquad
  +\frac{\Delta t}{2}\int_\Omega
  \nabla\cdot w gD^n \diff x, \quad \forall w\in V_h^1(\Omega), \\
  u^a & = \Phi^u(u^d; u^{n+1/2}), \\
  D^{n+1} & = \Phi^D(D^d; u^{n+1/2}), \\ \nonumber
  \int_\Omega w\cdot u^{n+1} \diff x & =
  \int_\Omega w\cdot u^a \diff x +
  \frac{\Delta t}{2}\int_\Omega w \cdot f(u^{n+1})^{\perp}\diff x \\
  & \qquad
  + \frac{\Delta t}{2}\int_\Omega
 (\nabla\cdot w) gD^{n+1} \diff x, \quad \forall w\in V_h^1(\Omega),
\end{align}
where $\Phi^u(v; u^{n+1/2})$ is the application of a chosen velocity
advection scheme over one timestep, with initial condition $v$ and
using advecting velocity $u^{n+1/2}$, whilst $\Phi^D$ is the same
thing but for layer depth. A Picard iteration to obtain $u^{n+1}$,
$D^{n+1}$ requires the solution of the following linear system
for the iterative corrections to $u^{n+1}$ and $D^{n+1}$, 
\begin{align}
  \int_\Omega w\cdot \Delta u \diff x
  + \frac{\Delta t}{2}\int_\Omega w\cdot f\Delta u^\perp\diff x & \nonumber \\
\label{eq:delta u}
  \qquad  - \frac{\Delta t}{2}\int_\Omega (\nabla\cdot w) g\Delta D \diff x
  & = -R_u[w], \quad \forall w\in V_h^1(\Omega), \\
  \int_\Omega \phi \Delta D \diff x + \frac{\Delta t}{2}
  \int_\Omega \phi (\nabla\cdot H )\Delta u \diff x & = -R_D[\phi],  \quad \forall \phi\in V_h^2(\Omega),
\label{eq:delta D}
\end{align}
where $R_u[w]$ and $R_D[\phi]$ are residuals for the implicit system,
and $H$ is the value of $D$ in a state of rest. After obtaining
$\Delta u$ and $\Delta D$, we replace $u^{n+1}\mapsto u^{n+1}+\Delta
D$, $D^{n+1}+\Delta D$, and repeat the iterative procedure again for a
fixed number of times. In staggered finite difference models, the
standard approach to solving Equations (\ref{eq:delta u}-\ref{eq:delta
  D}) is to neglect the Coriolis term, and then to eliminate $u$ to obtain a
discrete Helmholtz equation for which smoothers such as SOR are
convergent, and therefore the system is amenable to multigrid methods
or Krylov subspace methods (or a combination of the two). This
approach is problematic in the compatible finite element setting
because the velocity mass matrix is globally coupled and thus has a
dense inverse.

An alternative approach, which also allows us to keep the Coriolis
term in the linear system, is called hybridisation; this technique
dates back to the 1960s in engineering use, and to the late 1970s and
1980s in terms of numerical analysis \cite{boffi2013mixed}.  In this
approach, we consider solutions $u$ in the broken discontinuous space
$\hat{V}_h^1(\Omega)$. Then, Lagrange multipliers are introduced to enforce
continuity of the normal component of $u$, to force it back into
$V_h^1(\Omega)\subset\hat{V}_h^1(\Omega)$. We define a trace space for the Lagrange
multipliers, $W$, consisting of piecewise polynomials on each edge
$e\in \Gamma$, with no continuity requirements between edges, and so
the degree of the polynomials is the same as the degree of the
normal component of functions from $V_h^1(\Omega)$. The resulting system is
\begin{align}
  \int_\Omega w\cdot \Delta u \diff x \nonumber
  + \frac{\Delta t}{2}\int_\Omega w\cdot f\Delta u^\perp\diff x & \\
  \qquad  - \frac{\Delta t}{2}\int_\Omega (\nabla\cdot w) g\Delta D \diff x
   + \int_\Gamma \lambda \jump{w}\diff S
  & = -R_u[w], \quad \forall w\in \hat{V}_h^1(\Omega), \\
  \int_\Omega \phi \Delta D \diff x + \frac{\Delta t}{2}
  \int_\Omega \phi(\nabla\cdot H) \Delta u \diff x & = -R_D[\phi],  \quad \forall \phi\in V_h^2(\Omega), \\
  \int_\Gamma \mu \jump{u} \diff S & = 0, \quad \forall \mu\in W.
\label{eq:hybrid jump}
\end{align}
We see that if $u$, $D$, $\lambda$ solve this system of equations,
then $u\in V_h^1(\Omega)$, through Equation \eqref{eq:hybrid jump}. Further, if
we take $w\in V_h^1(\Omega)\subset \hat{V}_h^1(\Omega)$, then the $\lambda$ term vanishes
and so $u$, $D$ also solves Equations (\ref{eq:delta u}-\ref{eq:delta
  D}).  Further, since $\hat{V}_h^1(\Omega)$ and $V_h^2(\Omega)$ are \emph{both}
discontinuous spaces, they can be sparsely eliminated to obtain a
positive-definite system for $\lambda$. The resulting system is
symmetric and positive definite, so can be solved using the conjugate
gradient method preconditioned by Gauss-Seidel iteration, for example.
In addition, it can be shown that $\lambda$ is in fact an
approximation to $g\Delta D\Delta t/2$; this fact can then be used to
show that a multigrid method for the system for $\lambda$ will
converge \cite{gopalakrishnan2009convergent}. Having solved for
$\lambda$, $u$ and $D$ can be obtained by back-substitution separately
in each element. If the system for $\lambda$ has only been solved
approximately, then some form of projection (averaging values on each
side of element boundaries, for example) is required to obtain $\Delta
u\in V_h^1(\Omega)$. This approach to solving the linear system was explored in
the context of shallow water equations on the sphere in
\cite{shipton:_higher}. We are currently extending the implementation
and analysis to hybridisation of the linear system arising in three-dimensional compressible dynamical cores.

\section{Three-dimensional compatible finite element spaces}
\label{sec:3d}
The shallow water equations provide a useful stepping stone for
development of numerical methods for three-dimensional dynamical cores. We now
discuss the additional aspects that occur when building
discretisations in three dimensions. The extension of vertical grid
staggering to finite element methods, including high-order finite
elements, was introduced by \cite{gmd-2015-275}. Here, we consider the
use of three-dimensional compatible finite elements that have the same
structure in the vertical.

In analogy to the discussion of Section~\ref{sec:2D} we start with the following definition. 

\begin{definition}\label{def:3dspaces}
  Let $V_h^0(\Omega)\subset H^1(\Omega)$, $V_h^1(\Omega) \subset H(\mr{curl};\Omega)$ , $V_h^2(\Omega) \subset H(\ddiv;\Omega)$ and $V_h^3(\Omega) \subset L^2(\Omega)$ be a sequence of finite element spaces ($V_h^0(\Omega)$
  and $V_h^3(\Omega)$ contain scalar valued functions, whilst $V_h^1(\Omega)$ and $V_h^2(\Omega)$ contain
  vector valued functions). These spaces are called \emph{compatible}
  if:
  \begin{enumerate}
  \item $\nabla\psi \in V_h^1(\Omega)$, $\forall \psi \in V_h^0(\Omega)$,
  \item $\nabla\times v \in V_h^2(\Omega)$, $\forall v \in V_h^1(\Omega)$,
  \item $\nabla\cdot u \in V_h^3(\Omega)$, $\forall u \in V_h^2(\Omega)$, and
  \item there exist bounded projections $\pi^i$, $i=0,1,2,3$,
    such that the following diagram commutes.
  \begin{equation}
    \begin{tikzcd}[column sep=2em]
      H^1(\Omega)\arrow{r}{\nabla}\arrow{d}{\pi^0} &
      H(\rm{curl}; \Omega) \arrow{r}{\nabla\cdot}\arrow{d}{\pi^1}&
      H(\rm{div}; \Omega) \arrow{r}{\nabla\cdot}\arrow{d}{\pi^2}&
      L^2(\Omega) \arrow{d}{\pi^3} \\
      V_h^0(\Omega) \arrow{r}{\nabla}&
      V_h^1(\Omega) \arrow{r}{\nabla\times}&
      V_h^2(\Omega) \arrow{r}{\nabla\cdot}&
      V_h^3(\Omega)
    \end{tikzcd}
    \label{eq:3d diagram}
  \end{equation}
  \end{enumerate}
\end{definition}

In a compatible finite element discretisation of the three-dimensional
compressible Euler equations, we choose the velocity $u\in V_h^2(\Omega)$ and
the density $\rho \in V_h^3(\Omega)$. We will consider how to represent
temperature in a later section. It is an established fact in numerical
weather prediction that mesh cells must be aligned in vertical columns.
This is because the dynamics is very close to hydrostatic balance, which
is obtained by solving
\begin{equation}
\pp{p}{z} = -g\rho.
\end{equation}
In large scale models, the domain is very thin, so cells have much
greater extent in the horizontal than the vertical direction. If cells
are not aligned in columns, then the large horizontal errors in $p$
will be coupled into the hydrostatic equation, and the discretisation
generates spurious waves that rapidly degrade the state of hydrostatic
balance. We will see later that using a mesh that is aligned in columns
leads to well-defined hydrostatic pressures when compatible finite element
spaces are used. Hence, we consider meshes constructed from hexahedra
or triangular prisms. Where all cells have side walls
aligned in the vertical direction. We do allow for sloping cell tops
and bottoms, which facilitates terrain-following meshes.

Under these constraints, we build cells in the following way. We
construct a reference cell as the tensor product of a horizontal cell (a triangle or quadrilateral) and an interval. Then, each cell in the physical mesh is obtained by a mapping from the reference cell into $\mathbb{R}^3$, such that the above constraints are satisfied.

To build finite element spaces on this mesh, we first construct a
finite element on the reference cell. The finite element is formed
from tensor products of finite elements defined on the base triangle or square, and finite elements defined on the vertical interval. Of course, the square can also then be defined as the product of two intervals.  We then define finite element spaces on the mesh \emph{via} the pullback mappings listed in Table~\ref{tab:Fstar}. 
More precisely, if $F_K$ is a mapping from the
reference cell $\hat{K}$ to the physical cell $K$, then finite element
functions are related \emph{via}
\begin{itemize}
\item $\psi \in V_h^0(\hat{K}) \implies \psi\circ F_K^{-1} \in V_h^0(K)$,
\item $v \in V_h^1(\hat{K}) \implies J^{-T}\,v \circ F^{-1}_K \in V_h^1(K)$,
\item $u \in V_h^2(\hat{K}) \implies Ju/\det(J)\circ F_K^{-1} \in V_h^2(K)$,
\item $\rho \in V_h^3(\hat{K}) \implies \rho/\det(J)\circ F_K^{-1} \in V_h^3(K)$,
\end{itemize}
where $J$ is the Jacobian matrix $DF_K$.
Then, as in the two-dimensional case, given a decomposition $\mc T_h$ of the domain $\Omega$ in elements $K$, the finite element space $V^k_h(\Omega)$ on $\mc T_h$ is defined by requiring $V^k_h(K) = F^{-1*}_K(V^k_h(\hat{K}))$ and $V^k_h(\Omega)\subset HV^k(\Omega)$.

A key problem is that for the
velocity space $V^2_h(\Omega)$, this transformation results in non-polynomial
functions when the transformation is not affine (a linear
transformation composed with translation). We obtain non-affine
transformations whenever we use spherical shell
(\emph{i.e.}\ atmosphere-shaped) domains, or terrain-following coordinates,
whether we have triangular prismatic or hexahedral cells. This problem was discussed and the general loss of error quantified using FEEC in the case of compatible finite element spaces on cubical elements in \cite{Arnold14}. Here, we extend this discussion to cover the triangular prismatic cells that we wish to use in horizontally unstructured mesh numerical weather prediction. We also show that this
general loss of error is still possible for columnar meshes. For the
case of manifolds embedded in $\mathbb{R}^n$, this problem was
addressed by \cite{Holst12}, where it was shown that if the
computational mesh is a piecewise polynomial approximation of a smooth
manifold, then functions can be approximated at the optimal rate, and
further that numerical solutions of mixed Hodge Laplacian problems
converge optimally. Here, we extend these results to spherical shell
domains in the case of extruded meshes where the base mesh is either a
cubed sphere mesh of quadrilaterals, or an arbitrarily structured mesh
of triangles. This extension makes use of a transformation of the
sphere domain to a three-dimensional submanifold of $\mathbb{R}^4$, for
which an affine mesh is possible.


The rest of this section is structured as follows. In
Section \ref{sec:tfamilies} we describe the construction of tensor
product finite element spaces, including the extension to triangular prism elements. In Section~\ref{sec:curvielem} we prove some approximation properties of these spaces. In Section \ref{sec:glotran}, we discuss global transformations of meshes and analyse how these affect the approximation properties of finite element spaces. This analysis is then applied in Section \ref{sec:shell} to the case of a spherical annulus.

\subsection{Tensor product finite element spaces}
\label{sec:tfamilies}
The construction of tensor product compatible spaces is addressed in \cite{Arnold13,Arnold14} with application to cubic meshes. However, the procedure for the definition of such spaces is rather general and applies to any mesh whose elements are generated as Cartesian product of simplices. This covers both cell types that we are interested in, triangular prisms and hexahedra. The construction relies on having two simplices  $T\subset\mathbb{R}^n$ and $S\subset \mathbb{R}^m$ and a polynomial complex of compatible finite element spaces on each of these. Then one can produce a polynomial complex of compatible finite element spaces on $ T\times S$ using the ones defined on  $T$ and $S$.
Here we only show the result of this procedure for our case of interest, that is when $n=2$, $m=1$ and $T\times S$ is a prism. 

We start by selecting polynomial spaces on $S$ and $T$ on which we base the tensor product construction. Since $S$ is one-dimensional, we can only take on $S$ the sequence of spaces ($CG_s$,$DG_{s-1}$), for $s\geq1$, as this is the only valid choice of compatible polynomial spaces.  
As for the complex on $T$ we have more freedom. Here, we restrict ourself to the examples discussed in Section~\ref{sec:2D}. In particular, if we consider on $T$ the sequence of compatible spaces ($CG_r$,$RT_{r-1}$, $DG_{r-1}$) we obtain the following  compatible spaces on $T\times S$,
\begin{align*}
\mc{E}_{r,s}^{-}V^0(T\times S) &= CG_r(T) \otimes CG_s(S),\\
\mc{E}_{r,s}^{-}V^1(T\times S) &= [RT_{r-1}(T) \otimes CG_s(S)] \oplus [CG_r(T) \otimes  DG_{s-1} (S)\,\hat{z}],\\
\mc{E}_{r,s}^{-}V^2(T\times S) &= [DG_{r-1} (T)\, \hat{z} \otimes  CG_s(S)] \oplus [RT_{r-1}(T) \times  DG_{s-1} (S)\,\hat{z}],\\
\mc{E}_{r,s}^{-}V^3(T\times S) &= DG_{r-1}(T) \otimes DG_{s-1}(S), 
\end{align*}
where $\hat{z}$ is a unit vector field oriented in the extrusion direction. Alternatively, if we use the sequence ($CG_r$,$BDM_{r-1}$, $DG_{r-2}$) complex on $T$ we obtain the following compatible spaces on $T\times S$
\begin{align*}
\mc{E}_{r,s}V^0(T\times S) &= CG_r(T) \otimes CG_s(S),\\
\mc{E}_{r,s}V^1(T\times S) &= [BDM_{r-1}(T) \otimes CG_s(S)] \oplus [CG_r(T) \otimes  DG_{s-1} (S)\, \hat{z}],\\
\mc{E}_{r,s}V^2(T\times S) &= [DG_{r-2} (T)\, \hat{z} \otimes  CG_s(S)] \oplus [BDM_{r-1}(T) \times  DG_{s-1} (S)\,\hat{z}],\\
\mc{E}_{r,s}V^3(T\times S) &= DG_{r-2}(T) \otimes DG_{s-1}(S).
\end{align*}

Note that in both cases these spaces have been defined in the way they would appear in the complex, \emph{i.e.}\ $d^k: \mc{E}_{r,s}V^k \rightarrow \mc{E}_{r,s}V^{k+1}$ and the same for the $\mc{E}_{r,s}^-V^k$ spaces. The fact that such spaces can be called compatible follows from the possibility to define bounded projections in accordance to Definition~\ref{def:3dspaces}. This can indeed be done by exploiting the tensor product structure
as shown in \cite{Arnold14}.
More details on the definition and the implementation of the tensor product spaces discussed in this section can also be found in  \cite{mcrae2014automated}.

\subsection{Approximation properties and curvilinear elements}
\label{sec:curvielem}
In this section, we analyse the convergence properties of the finite element spaces defined in Section \ref{sec:tfamilies} for specific classes of maps $F_K$.
In particular, we assume that the mesh is generated by either affine or multi-linear transformations and, for the resulting finite-dimensional space $V^k_h(\Omega)$, we are interested in providing bounds for the quantity
\begin{equation}\label{eq:best}
\inf_{p\in V^k_h(\Omega)}\|u-p\|_{L^2(\Omega)}.
\end{equation}
As for non-affine meshes, we will mostly use the results in \cite{Arnold14}, which focus on tensor product spaces on curvilinear cubic elements, and extend them to general tensor product elements. 

We start by recalling that a finite element space is constructed using a set of functions on a reference element $\hat{K}$ and by a set of maps $F_K:\hat{K} \rightarrow K$, where $K$ is an element in the domain decomposition $\mc T_h$. The maps $F_K$ can be arbitrary apart from some regularity conditions. These are determined by the way the pullback of $F_K$ transforms under Sobolev norms, which is described in the following theorem. 
\begin{theorem}\label{th:bound}
{(\cite{Arnold14}, Theorem 2.1)}
Let $M$ be  a positive constant and $r$ a non-negative integer. There exists a constant $C$ depending only on $r$, $M$ and the dimension of the domain $n$, such that
\begin{equation}
\|F_K^*v\|_{H^rV^k(\hat{K})} \leq C \|v\|_{H^rV^k({K})}, \quad v \in H^rV^k(K),
\end{equation}
whenever $\hat{K}$ is a domain in $\mathbb{R}^n$ and $F_K:\hat{K}\rightarrow K$ is a $C^{r+1}$ diffeomorphism satisfying
\begin{equation}
\max_{1\leq s\leq r+1} |F_K|_{W^s_\infty(\hat{T})} \leq M , \quad
|F_K^{-1}|_{W^1_\infty(\hat{T})}\leq M . 
\end{equation}
\end{theorem} 

In order to get an estimate for the quantity in Equation~\eqref{eq:best}, the Cl\'ement interpolant construction is generally invoked. First, for any $q$-dimensional domain $Q$ define $\mc P_r (Q)$ to be the space of polynomial scalar functions on $Q$ with degree up to $r$. Moreover, let $\mc P_rV^k(Q)$ be equal to $\mc P_r (Q)$ if $k=0,q$, or to the space of vector fields with coefficients in $\mc P_r (Q)$ if $0<k<q$.
If we let $V^k_h(K)$ be the restriction of the finite element space $V^k_h(\Omega)$ to the element $K$ and assume $\mc P_r V^k(K)\subset V^k_h(K)$, then one can design a projection $\pi:H^s V^k(\Omega) \rightarrow V^k_h(\Omega)$, called the Cl\'ement interpolant, for any $(n-k)/2<s \leq r+1$, such that $\pi|_K u = (\pi u)|_K$ for any $u \in H^s V^k(\Omega)$ and
\begin{equation}\label{eq:locest}
\|u-\pi u\|_{L^2({K})} \leq C h_K^{s}|u|_{H^{s}V^k({K})}, \quad u \in {H^{s}V^k({K})}\, ,
\end{equation}
where $h_K$ denotes the diameter of $K$. Summing up over all elements both sides of Equation~\eqref{eq:locest} yields the global estimate
\begin{equation}\label{eq:gloest}
\|u-\pi u\|_{L^2(\Omega)} \leq C h^{s} \sum_{T\in \mc T_h}|u|_{H^{s}V^k(T)} =C h^{s}|u|_{H^{s}V^k(\Omega)}, \quad u \in {H^{s}V^k(\Omega)},
\end{equation}
where $h$ is the maximum element diameter in the decomposition $\mc T_h$ of $\Omega$. Clearly, then, Equation~\eqref{eq:gloest} implies the same estimate for the best approximation error in Equation~\eqref{eq:best}.

Therefore, for a given choice of the function space $V^k_h(\hat{K})$ on the reference element $\hat{K}$, one is led to determine the conditions under which $\mc P_r V^k(K)\subset V^k_h(K)= F^{-1*}_K(V_h^k(\hat{K}))$ in order for the estimates in Equation~\eqref{eq:gloest} to hold.
We carry out this analysis for the polynomial spaces defined in Section~\ref{sec:tfamilies}, restricting the choice of $F_K$ to affine and multilinear transformations as in \cite{Arnold14}. The results are expressed in the following theorems which essentially are a restatement of Theorem 6.1 in \cite{Arnold14}, to the specific spaces considered here. 

We assume that the mesh is constituted by tensor product elements, so that we look for a polynomial space on the reference element $\hat{K}$ belonging to either one of the families $\mc E_{r,s}$ or $\mc E_{r,s}^- $ such that the estimates in Equations~\eqref{eq:locest} and \eqref{eq:gloest} hold. The sufficient conditions for convergence at a given rate relative to the family $\mc E_{r,s}$ are stated in the following theorem.

\begin{theorem}\label{th:extruded}
Consider a three-dimensional element $\hat{K}=\hat{T}\times\hat{S}$ obtained as tensor product of a two-dimensional simplex $\hat{T}$ and a one-dimensional simplex $\hat{S}$.
Let $F_K$ be an affine map and $\mc E_{r+\nu,r+\mu} V^k(\hat{K})\subset V^k_h(\hat{K})$, where $\nu = \min(k,2)$ and $\mu = \min(k,1)$; or let $F_K$ be a multilinear map and $\mc  E_{2(r+k),r+k}V^k(\hat{K})\subset V^k_h(\hat{K})$. Then $\mc P_rV^k(K)\subset V^k_h(K)$ and Equations~\eqref{eq:locest} and \eqref{eq:gloest} hold.
\end{theorem}

Finally, we consider a specific class of transformations that leave the coordinates on one of the two simplices $\hat{T}$ or $\hat{S}$  invariant up to an affine transformation. This corresponds to the case of local area models with terrain-following meshes.
Specifically, we assume $F_K = A_K \circ D_K$ where $D_K:\hat{K}\rightarrow {\tilde{K}}$ is a special multilinear deformation of the reference element and $A_K:\tilde{K}\rightarrow K$ is an arbitrary affine map. The map $D_K$ is defined in coordinates by $D_K =(I_T,M_S)$ with $I_T$ being the identity map on $\hat{T}$ and $M_S:\hat{K}\rightarrow{\mathbb{R}^m}$ being a sufficiently smooth multilinear map. We call the maps $F_K$ constructed in this way as multilinear maps invariant on $\hat{T}$. Analogously, if $D_K =(M_T,I_S)$ with $I_S$ being the identity map on $\hat{S}$ and $M_T:\hat{K}\rightarrow{\mathbb{R}^n}$ being a sufficiently smooth multilinear map, we call the relative $F_K$ as a multilinear map invariant on $\hat{S}$. A graphical representation of such constructions is given in Figure~\ref{fig:scheme1}.
The sufficient conditions for convergence at a given rate relative to these classes of transformations and for the family $\mc E_{r,s}$ are stated in the following theorem.

\begin{figure}[h]
\centering
\begin{tikzpicture}
    \node[anchor=south west,inner sep=0] at (0,0)
 {\centering    
 \includegraphics[scale=0.6]{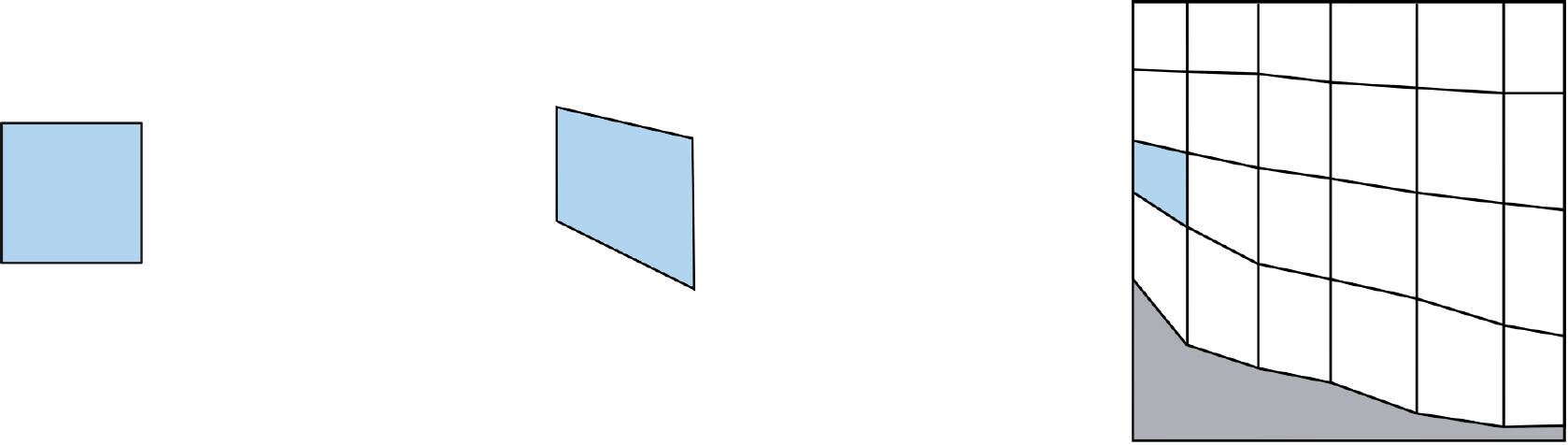}};
      \draw [thick, ->] (4.9,1.5) to [out=-20,in=200] (7,1.5);
       \draw [thick, ->] (1.2,1.5) to [out=-20,in=200] (3.3,1.5);
       \draw [thick, ->] (1.2,2)to [out=25,in=155] (7,2);
       \draw [thick] (0.05,1)to [out=-70,in=110] (0.475,0.9);
       \draw [thick] (0.475,0.9) to [out=70,in=-110] (0.9,1)  ;
       \draw [thick] (-0.18,1.23)to [out=+160,in=-20] (-0.28,1.65);
       \draw [thick] (-0.28,1.65)to [out=20,in=-160] (-0.18,2.07);
     \node at (0.475,0.55) {$\hat{T}$};
     \node at (-0.65,1.65) {$\hat{S}$};
     \node at (6,0.9) {$A_K$};
     \node at (2.2,0.9) {$D_K$};
     \node at (4.1,3) {$F_K$};
\end{tikzpicture}\caption{A two-dimensional mesh generated by multilinear transformations invariant on $\hat{T}$ of a reference element (on the left).}\label{fig:scheme1}
\end{figure}
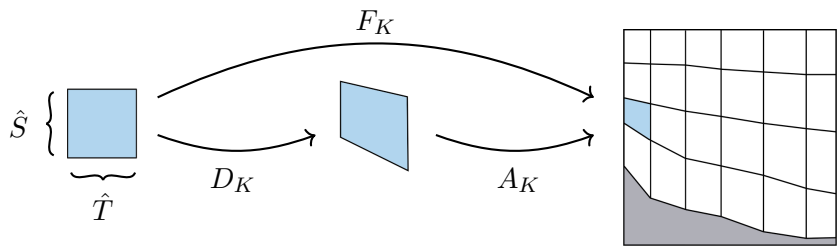

\begin{theorem}\label{th:extrudedinv}
Consider a three-dimensional element $\hat{K}=\hat{T}\times\hat{S}$ obtained as tensor product of a two-dimensional simplex $\hat{T}$ and  a one-dimensional simplex $\hat{S}$.
Let $F_K$ be a multilinear map invariant on $\hat{T}$, $\mu = \min(k,1) $ and $\mc E_{2r+\mu+k,r+\mu}V^k(\hat{K})\subset V^k_h(\hat{K})$; or let $F_K$ be a multilinear map invariant on $\hat{S}$, $\nu = \min(k,2) $ and $\mc E_{2(r+\nu),r+k}V^k(\hat{K})\subset V^k_h(\hat{K})$. Then $\mc P_rV^k(K)\subset V^k_h(K)$ and Equations~\eqref{eq:locest} and \eqref{eq:gloest} hold.
\end{theorem}

As one would expect, when $F_K$ is a multilinear map invariant on $\hat{T}$ (or $\hat{S}$) we obtain sufficient conditions which are  considerably less strict than the ones relative to general multlinear maps. However, this is true only when $k>1$ (or $k>2$ respectively) as the conditions that apply for lower values of $k$ are the same in both cases.

To summarise the results of this section, we can state that the spaces generated using the function spaces $V^k_h(K)= F_K^{-1*}(\mc E_{r,s} V^k(\hat{K}))$ yield the convergence rates shown in Table~\ref{tab:convrat}. Note that multilinear transformations invariant on the one-dimensional simplex $\hat{T}$ are precisely the type of transformations that occur in atmosphere simulations when dealing with meshes following a varying topography.
The convergence rates relative to the spaces generated using the function spaces $V^k_h(K)= F_K^{-1*}(\mc E_{r,s}^- V^k(\hat{K}))$ are different, but can be derived by determining the smallest polynomial degrees $p$ and $q$ such that $\mc E_{r,s}^- V^k(\hat{K})\subseteq \mc E_{p,q} V^k(\hat{K})$, and then using the results above.

\begin{table}\begin{center}
\renewcommand{\arraystretch}{1.2} 
\begin{tabular}{c|llll}
\hline\hline
  \multicolumn{1}{c|}{\multirow{2}{*}{$\mc E_{r,s}V^k(\mc T_h)$}} & \multicolumn{1}{c}{\multirow{2}{*}{Affine}} & \multicolumn{1}{c}{\multirow{2}{*}{Multilinear}} & \multicolumn{1}{c}{Multilinear} \\
  & & & \multicolumn{1}{c}{invariant on $\hat{T}$}\\ 
\hline
$k=0$  & ${\mr{min}}(r,s)+1$  & ${\mr{min}} (\lfloor r/2 \rfloor,s)+1$    &  ${\mr{min}} (\lfloor r/2 \rfloor,s)+1$  \\[2pt]
$k=1$  & ${\mr{min}}(r,s)$  & ${\mr{min}} (\lfloor r/2 \rfloor,s)$      &  ${\mr{min}} (\lfloor r/2 \rfloor,s)$  \\[2pt]
$k=2$  & ${\mr{min}}(r-1,s)$  & ${\mr{min}} (\lfloor r/2 \rfloor,s)-1$     &  ${\mr{min}} (\lfloor (r-1)/2 \rfloor,s)$  \\[2pt]
$k=3$  & ${\mr{min}}(r-1,s)$  & ${\mr{min}} (\lfloor r/2 \rfloor,s)-2$   &   ${\mr{min}} (\lfloor (r-2)/2 \rfloor,s)$   \\[2pt]
\hline
\end{tabular}
\caption{Convergence rates for finite element spaces obtained by pull-back of $\mc E_{r,s}V^k(\hat{T}\times \hat{S})$ with $\mr{dim}(\hat{T}) =2$ and $\mr{dim}(\hat{S}) =1$ (triangular prism elements), using different classes of transformations $F_K$. The symbols $\lfloor x \rfloor$ denotes the largest integer smaller or equal to $x$.}
\label{tab:convrat}
\end{center}\end{table}

\subsection{Approximation properties for global transformations}
\label{sec:glotran}

The last subsection was devoted to the study of the approximation properties of finite element spaces defined on meshes generated by arbitrary local transformations of a reference element. In several applications, however, the mesh is generated by means of a global transformation from a reference mesh on which spaces with optimal convergence properties are defined. In these cases, if the global map is sufficiently regular, one can transfer the convergence properties of the original mesh to the deformed one.

Numerical weather prediction models require the capability to solve in
a spherical annulus domain; we are also interested in domains with
topography and terrain-following coordinates. These are not polygonal
domains, and we must model them using piecewise-polynomial
approximations $\tilde{\Omega}$. With this in mind, let
$\tilde{\Omega}$ be also an $n$-dimensional Riemmanian manifold, such
that there exists a $C^1$ diffeomorphism $F:\Omega\rightarrow
\tilde{\Omega}$. We want to show how elements of the spaces $HV^k(\Omega)$ can
be transformed by $F$ into elements of $HV^k(\tilde{\Omega})$. It is
understood that simple composition with $F^{-1}$ is not enough for
this purpose. A well-known example for this is the case of the space
$H(\mr{div};\Omega)$, where the correct transformation is given by the
contravariant Piola transform which defines an isomorphism between
$H(\mr{div};\Omega)$ and $H(\mr{div};\tilde{\Omega})$.

It is easy to check that one needs to define appropriate
transformations, analogous to the Piola transform, for each space
$HV^k(\Omega)$. We denote such transformations with a unique
symbol, $F^*: HV^k(\tilde{\Omega})\rightarrow HV^k(\Omega)$, namely
the pullback of $F$, with inverse $F^{-1*}: HV^k(\Omega)\rightarrow HV^k(\tilde{\Omega})$, which defines an isomorphism between
$HV^k(\Omega)$ and $HV^k(\tilde{\Omega})$. $F^{-1*}$ coincides with the contravariant Piola transform when applied to $HV^{n-1}(\Omega)$, and takes a different
form depending on the degree $k$ in the complex and the dimension of
the domain $n$. The action of $F^{-1*}$ on an element $v \in HV^k(\Omega)$
is defined in Table~\ref{tab:Fstar}.  Such definitions allow to
preserve the Hilbert complex structure, \emph{i.e.}\ the following diagram
commutes 
\begin{table}\begin{center}
\setlength{\tabcolsep}{8pt}
\begin{tabular}{l|llll}
\hline\hline\\[-10pt]
   & $k=0$ & $k=1$ & $k=2$ & $k=3$ \\
   \hline\\[-10pt]
$n=1$  &  $v \circ F^{-1}$  & $v/\mr{det}(J) \circ F^{-1}$     \\[2pt]
$n=2$ & $v \circ F^{-1}$ & $J\, v / \mr{det}(J)  \circ F^{-1}$     &  $v/\mr{det}(J) \circ F^{-1}$  \\[2pt]
$n=3$  & $v \circ F^{-1}$  & $J^{-T}\,v \circ F^{-1}$   &   $J\, v / \mr{det}(J)  \circ F^{-1}$  & $v/\mr{det}(J) \circ F^{-1}$ \\[2pt]
\hline
\end{tabular}
\caption{$F^{-1*}v$ for $1\leq n \leq 3$ and $0\leq k \leq 3$. $J$ is the Jacobian matrix $DF$ of the map $F$.}
\label{tab:Fstar}
\end{center}\end{table}

\begin{equation}\label{eq:pulcom}
\begin{tikzcd}[column sep=2em]
0\arrow{r}&
HV^0(\Omega)\arrow{r}{\ed^0}\arrow{d}{F^{-1*}} &
HV^1(\Omega)\arrow{r}{\ed^1}\arrow{d}{F^{-1*}}&
\ldots \arrow{r}{\ed^{n-1}}&
HV^n(\Omega) \arrow{r}\arrow{d}{F^{-1*}}&
0 \\
0 \arrow{r}&
HV^0(\tilde{\Omega}) \arrow{r}{\ed^0}&
HV^1(\tilde{\Omega}) \arrow{r}{\ed^1}&
\ldots\arrow{r}{\ed^{n-1}}&
HV^n(\tilde{\Omega}) \arrow{r}& 0
\end{tikzcd}
\end{equation}
Note that in the diagram above the differential operators are expressed accordingly to the metric defined on $\Omega$ or $\tilde{\Omega}$.

Consider two $n$-dimensional manifolds $\tilde{\Omega}$ and $\Omega$ and a decomposition ${\mc T}_{\tilde{h}}$ of $\tilde{\Omega}$ by either simplices or tensor product elements. In other words, each element $\tilde{K}$ in ${\mc T}_{\tilde{h}}$ is generated via a sufficiently regular transformation $F_{\tilde{K}}:\hat{K}\rightarrow \tilde{K}$ as in Section~\ref{sec:curvielem}, where $\hat{K}$ is a reference element. Let $G$ be a diffeomorphism from $\tilde{\Omega}$ to $\Omega$, and let $G_K:\tilde{K}\rightarrow K$ be the restriction of $G$ to $\tilde{K}$. This defines a decomposition $\mc T_h$ of $\Omega$ in which each element $K$ is given by a transformation  of the reference element $\hat{K}$ defined by $F_K:= G_K \circ F_{\tilde{K}}$. Then one can define a compatible function spaces on each $K$ as pull-back of the ones defined on $\hat{K}$. In particular, given a space $V^k(\hat{K})$ on the reference element  $\hat{K}$, the corresponding space on $K$ is given by $V^k(K):=G^{-1*}_{K}F^{-1*}_{\tilde{K}} V^k(\hat{K})$.

Let $V^k_{\tilde{h}}(\tilde{\Omega})$ and $V^k_h(\Omega)$ be the resulting finite element spaces on $\tilde{\Omega}$ and $\Omega$ respectively.
The definition of the space $V^k_h(\Omega)$ makes sense only if $V^k_h(\Omega)\subset HV^k(\Omega)$. For this to hold, it is sufficient that $G$ is $C^0$ at the element interfaces and $G^{-1*}_K$ is bounded inside each element. These conditions, however, are not sufficient for optimal convergence. The following theorem establishes the level of regularity of $G$ required to guarantee optimal convergence rates for the relative finite element spaces on $\Omega$.

\begin{theorem}\label{th:gtransf}
Assume that $\mc P_r V^k(\tilde{K})\subset V^k_h(\tilde{K})$ and let $G:\tilde{\Omega}\rightarrow\Omega$ be a $C^0$ diffeomorphism. Moreover, assume that for any $\tilde{K}\in {\mc{T}}_{\tilde{h}}$ the restriction $G_{K}:\tilde{K} \rightarrow K$ is a $C^{r+2}$ diffeomorphism, such that
\begin{equation}\label{eq:constm}
\max_{1\leq s\leq r+2} |G_K|_{W^s_\infty(\tilde{\Omega})} \leq M , \quad
|G_K^{-1}|_{W^1_\infty(\tilde{\Omega})}\leq M,
\end{equation}
for $M>0$. Then, for $(n-k)/2<s\leq r+1$,
\begin{equation}\label{eq:finest}
\inf_{p\in V^k_h(\Omega)} \| u-p \|_{L^2(\Omega)} \leq C \tilde{h}^{s} \| u \|_{H^{s}V^k(\Omega)}, \quad u \in H^sV^k(\Omega),
\end{equation}
where $C>0$ and $\tilde{h}$ is the largest diameter of the elements in the decomposition ${\mc T}_{\tilde{h}}$ of $\tilde{\Omega}$.
\end{theorem}

\begin{proof}
By definition any function $p\in V^k_h(\Omega)$ can be expressed formally as $p= G^{-1*} \tilde{p}$ with $\tilde{p}\in V^k_{\tilde{h}}(\tilde{ \Omega})$, with the pull-back computed element-wise. Therefore,
\begin{equation}\label{eq:th1}
\inf_{p\in V^k_h(\Omega)} \| u-p \|_{L^2(\Omega)}
=\inf_{\tilde{p}\in V^k_{\tilde{h}}(\tilde{ \Omega})} \|G^{-1*}( G^{*} u-\tilde{p}) \|_{L^2(\Omega)}.
\end{equation}
Using Theorem~\ref{th:bound} and Equation~\eqref{eq:gloest}, we obtain the following bound,
\begin{equation}\label{eq:th2}
\begin{split}
\inf_{\tilde{p}\in V^k_{\tilde{h}}(\tilde{ \Omega})} \|G^{-1*}( G^{*} u-\tilde{p}) \|_{L^2({\Omega})}&\leq 
C \inf_{\tilde{p}\in V^k_{\tilde{h}}(\tilde{ \Omega})} \| G^{*} u-\tilde{p}\|_{L^2(\tilde{\Omega})}\\
&\leq   C \tilde{h}^s\,\sum_{\tilde{K}\in \mc T_{\tilde{h}}} | G_K^{*} u|_{H^{s}V^k(\tilde{K})}\, ,
\end{split}
\end{equation}
for $(n-k)/2<s\leq r+1$. Applying again Theorem~\ref{th:bound} to the right hand side of Equation~\eqref{eq:th2}, we obtain
\begin{equation}\label{eq:th3}
| G^{*}_K u|_{H^{s}V^k(\tilde{K})}
\leq \| G_K^{*} u\|_{H^{s}V^k(\tilde{K})}
\leq C \| u\|_{H^{s}V^k({K})}.
\end{equation}
The proof is completed by combining Equations~\eqref{eq:th1}, \eqref{eq:th2} and \eqref{eq:th3}.
\end{proof}


The estimates derived in this section might seem to have little utility given that functions in $V^k_h(K)$ cannot be computed accurately for arbitrary diffeomorphisms $G$. Nonetheless, Theorem~\ref{th:bound} finds an important application in the context of problems defined on surfaces. To see this, let $\Omega$ be a $n$-dimensional manifold embedded in $\mathbb{R}^{n+1}$. Consider a family of decompositions of $\Omega$ by piecewise linear elements parameterized by the maximum element diameter $\tilde{h}$. Such meshes implicitly define a family of approximate domains $\tilde{\Omega}_{\tilde{h}}$, on which it is possible to define as finite element space $V^k_{\tilde{h}}(\tilde{\Omega}_{\tilde{h}})$ any of the polynomial spaces considered above. Under some regularity assumption on $\Omega$ \cite{Holst12} one can also construct a family of diffeomorphisms $G_{\tilde{h}}:\tilde{\Omega}_{\tilde{h}}\rightarrow \Omega$, which by construction can only be $C^0$ at the element interfaces and that approach the identity map as 
$\tilde{h} \rightarrow 0$. 
The pull-backs $G_{\tilde{h}}^{-1*}$ can be used to construct finite element spaces $V^k_h(\Omega)$. Hence, a natural measure of the error between $G_{\tilde{h}}$ and the identity map on $\Omega$ is given by 
\begin{equation}\label{eq:errid}
\|\mathrm{adj}(G_{\tilde{h}}^{-1*}) \circ G_{\tilde{h}}^{-1*} - \mr{Id} \| \, ,
\end{equation}
where $\mr{adj}$ denotes the adjoint with respect to the $L^2$ inner product, $\mr{Id}$ is the identity map and $\|\cdot\|$ is the standard operator norm on the space of $L^2$ functions or vector fields, accordingly to the degree $k$.

The exact same setting for this problem and its connection with the FEEC are explored in \cite{Holst12} for the solution of elliptic problems. Their main conclusion is that even though one uses only functions in $V^k_{\tilde{h}}(\tilde{\Omega}_{\tilde{h}})$ as computational basis functions, the final error estimates depend on the quantity in Equation~\eqref{eq:errid} and on the best approximation error on the space $V^k_h(\Omega)$, as measured by 
\begin{equation}
\inf_{p\in V^k_h(\Omega)} \| u-p \|_{L^2(\Omega)}.
\end{equation}
Therefore, if Theorem~\ref{th:gtransf} holds, with the constant $M$ in Equation~\eqref{eq:constm} being independent of $\tilde{h}$, one can apply the results in \cite{Holst12} and obtain optimal convergence rates for the problem defined on the original surface $\Omega$.

\subsection{Finite element spaces in ``atmosphere-shaped'' domains}
\label{sec:shell}

In this section we consider as an example finite element spaces defined on a spherical shell in $\mathbb{R}^3$, which is a typical scenario for atmosphere simulations, and we compute their approximation properties. The spherical shell $C$ is the submanifold of $\mathbb{R}^3$ defined by
\begin{equation}
C:=\{ \mb x\in \mathbb{R}^3 \,|\, a \leq\| \mb x\|\leq b \}\, ,
\end{equation}
for any $b>a>0$. The shell $C$ is topologically the same as the generalised cylinder $\tilde{C} = S\times [a,b]\subset \mathbb{R}^4$, where $S$ is the sphere of radius $a$,
\begin{equation}
S:=\{ \mb x\in \mathbb{R}^3 \,|\, \| \mb x\|=a \}.
\end{equation}
In fact, we can construct a continuous transformation $F:\mathbb{R}^4\rightarrow\mathbb{R}^3$  such that the restriction $F|_{\tilde{C}}:\tilde{C} \rightarrow C$ is a $C^\infty$ diffeomorphism. A particular choice for $F$ is the multilinear map defined by
\begin{equation}
F(x_1,x_2,x_3,x_4) = (x_1,x_2,x_3) \left(1+\frac{x_4-a}{a}\right),
\end{equation}
where $\{x_i\}_{i=1}^3$ and $\{x_i\}_{i=1}^4$ are Cartesian coordinates in $\mathbb{R}^3$ and $\mathbb{R}^4$, respectively.

Given a decomposition of $[a,b]$ and one for $S$, one can construct a decomposition $\mc T_{\tilde{h}}$ for $\tilde{C}$ by tensor product, meaning that each element in $\mc T_{\tilde{h}}$ is constructed as tensor product element of a simplex in each mesh. The decomposition $\mc T_{\tilde{h}}$ implicitly defines the approximate cylinder $\tilde{C}_{\tilde{h}}$. The map $F$ restricted to $\tilde{C}_{\tilde{h}}$ defines a decomposition $\mc T_h$ for $C$ and the approximate shell $C_{h}$ as image of $F$ restricted to $\tilde{C}_{\tilde{h}}$. Any element in $K\in\mc T_h$ is constructed using a multilinear map $F|_{\tilde{K}}:\tilde{K}\rightarrow K$ with $\tilde{K}\in \mc T_{\tilde{h}}$.

Since the decomposition $\mc T_{\tilde{h}}$ is composed by tensor product elements we can define a finite element space $V^k_{\tilde{h}} (\tilde{C}_{\tilde{h}})$ on $\tilde{C}_{\tilde{h}}$ belonging to either one of the families $\mc E_{r,s}V^k$ or $\mc E^-_{r,s}V^k$. The map $F$ induces via pull-back a definition for a finite element space $V^k_{{h}} ({C}_{{h}})$ on ${C}_{{h}}$. Since the derivatives of the map $F|_{\tilde{K}}$  are bounded away from zero by a constant independent of the size of $\tilde{K}$, we can apply Theorem~\ref{th:extruded} and the classical Cl\'ement interpolant construction to obtain a first estimate for the approximation properties of the space $V^k_h(C)$. In this case, however, such estimate is not sharp. As a matter of fact, since $F|_{\tilde{C}_{\tilde{h}}}$ is certainly $C^0$ at the element interfaces Theorem~\ref{th:gtransf} also applies providing optimal convergence rates for the space $V^k_h(C_h)$ as inherited by the ones of   $V^k_{\tilde{h}}(\tilde{C}_{\tilde{h}})$.

As explained in the last section, the error estimates for $V^k_{h}(C_h)$ can be transferred to a finite element space defined on $C$, if we are able to define a diffeomorphism from $C_h$ to $C$ that satisfies the requirements of Theorem~\ref{th:gtransf}. 
We now show how this can be done.
Using the techniques described in \cite{Holst12}, one can construct a diffeomorphism $G_{\tilde{h}}: \tilde{C}_{\tilde{h}} \rightarrow \tilde{C}$ that converges to the identity map on $\tilde{C}$ as $\tilde{h}\rightarrow 0$. The map $G_{\tilde{h}}$ can be lifted via the map $F$ to define the diffeomorphism $G_{h}:C_{h} \rightarrow C$, \ie
\begin{equation}\label{eq:mapc}
G_h:= F|_{\tilde{C}}\circ G_{\tilde{h}} \circ  (F|_{\tilde{C}_{\tilde{h}}})^{-1}.
\end{equation}
In Figure~\ref{fig:scheme} a graphical representation of the maps is presented in the two-dimensional case. Because of the smoothness of $F$, if $G_{\tilde{h}}$ satisfies the hypothesis of Theorem~\ref{th:gtransf} so does $G_h$. Therefore, the optimal convergence rates can be extended to the finite element spaces defined on $C$ via $G_h$. As explained in Section~\ref{sec:glotran}, in order for these estimates to be useful $G_h$ needs to approach the identity as $\tilde{h}\rightarrow 0$, in the sense of Equation~\eqref{eq:errid}. Formally, this can be verified directly by letting $\tilde{h}\rightarrow 0$ in Equation~\eqref{eq:mapc}.

It is worth observing that given a  differential problem on $C$, high order convergence can be achieved just as long as $G_h$ approaches the identity sufficiently fast with mesh refinement. In other words, if the error term in Equation~\eqref{eq:errid} rewritten for $G_h$ scales as ${h}^{s+1}$, then $s+1$ is the highest convergence rate achievable regardless of the polynomial order of the spaces on $\tilde{C}_{\tilde{h}}$. In \cite{Holst12}, it is proven that if one adopts an isoparametric approach $s$ is the order of the polynomial approximation $\tilde{C}_{\tilde{h}}$ of the surface $\tilde{C}$. Therefore, if higher convergence rates are required, one just needs to adopt a higher order approximation of  $\tilde{C}$ instead of a piecewise linear one and repeat the considerations above. 

Numerical experiments were performed to verify the expected convergence rates for both direct $L^2$-projections and the Helmholtz problem \cite{Arnold06} on the shell  as shown in Figure~\ref{fig:convtest} and~\ref{fig:convtesth} respectively.

It should be stressed that the final estimates on the curved domain rely essentially on the possibility of defining the global map $F$. Given this, proving optimal convergence rates in other cases of interest boils down to constructing such maps. For example, since in the meteorological community the use of cubed sphere meshes with quadrilateral cells is widespread, it is reasonable to ask whether such meshes enjoy optimal approximation properties when using the FEEC. This meshes are constructed using the $C^0$ map from the surface of a cube to the sphere. This map possesses enough smoothness for Theorem~\ref{th:gtransf} to apply and can be easily extended to the shell case by means of an extrusion in the fourth dimension as described above.

A similar discussion to the one of the present section can be applied to other domains which are relevant for geophysical applications. For example, one is generally interested in performing simulations in a domain that approximates the actual shape of the Earth, including terrain irregularities such as continents and mountains. In this case, instead of modifying the map $F$ one can think of defining a domain $\tilde{C}\subset\mathbb{R}^4$ such that its image $C$ under $F$ coincides with the desired Earth atmosphere approximation. If $C$ is sufficiently smooth then $\tilde{C}$ is a regular surface in $\mathbb{R}^4$ and one can repeat the considerations of this section to generate a finite element space on $C$ with optimal convergence properties.
\begin{figure}[h]
\centering
 \begin{minipage}[t]{0.9\linewidth}
\begin{tikzpicture}
\centering
    \node[anchor=south west,inner sep=0] at (0,0)
 { \includegraphics[scale=0.6]{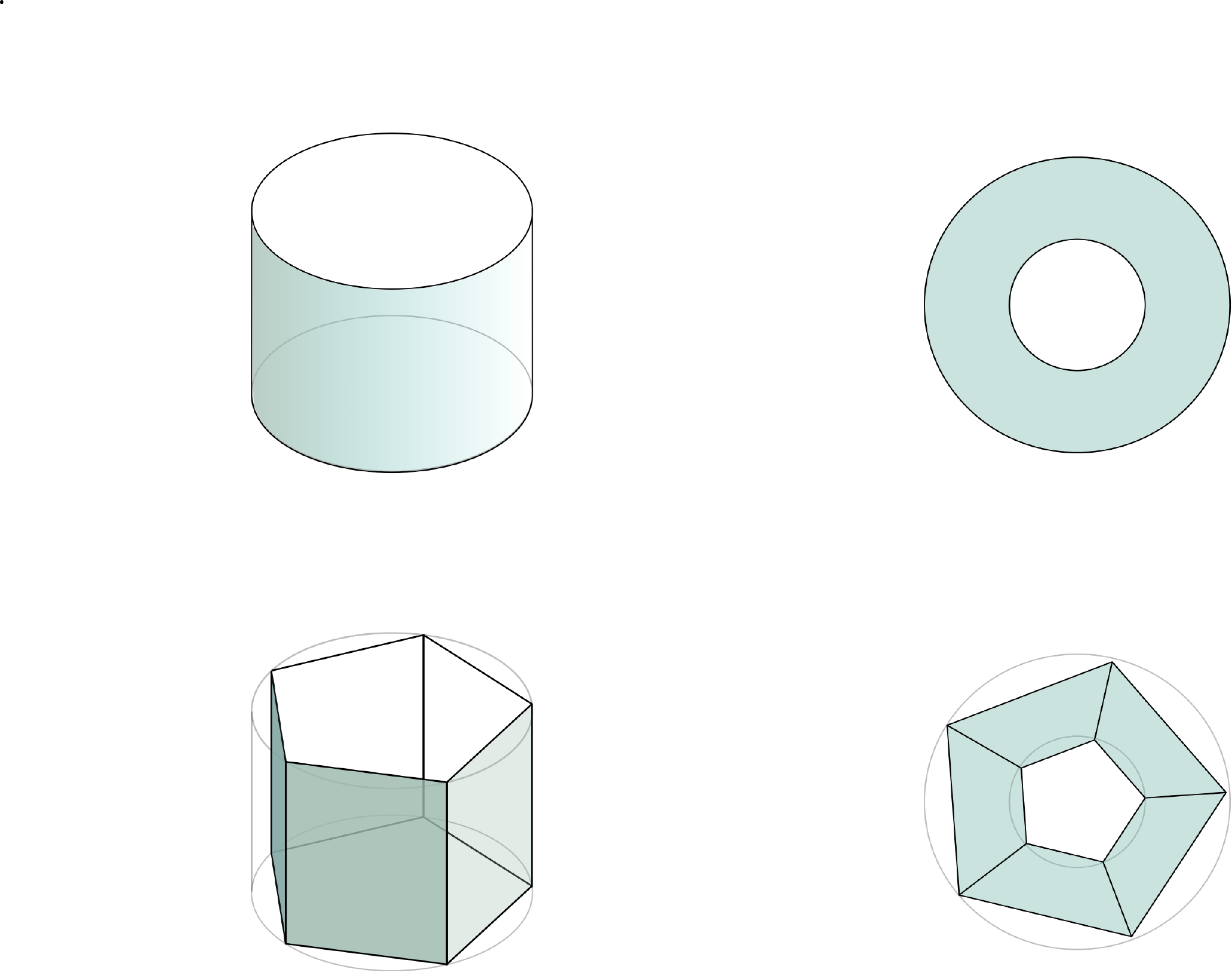}};
    \draw [ thick,->] (6,6.8) -- (9,6.8);
    \draw [thick, ->] (6,1.7) -- (9,1.7);
    \draw [thick, ->] (4.05,3.8) -- (4.05,4.8);
    \draw [thick, ->] (11.05,3.8) -- (11.05,4.8);
    \node at (7.5,7.2) {$F|_{\tilde{C}}$};
    \node at (7.5,2.1) {$F|_{\tilde{C}_{\tilde{h}}}$};
    \node at (3.55,4.2) {$G_{\tilde{h}}$};
  \node at (10.55,4.2) {$G_{h}$};
\end{tikzpicture}
\end{minipage}
\caption{Graphical representation of the maps involved in the construction of $G_h$ in the two-dimensional case.}\label{fig:scheme}
 
\end{figure}
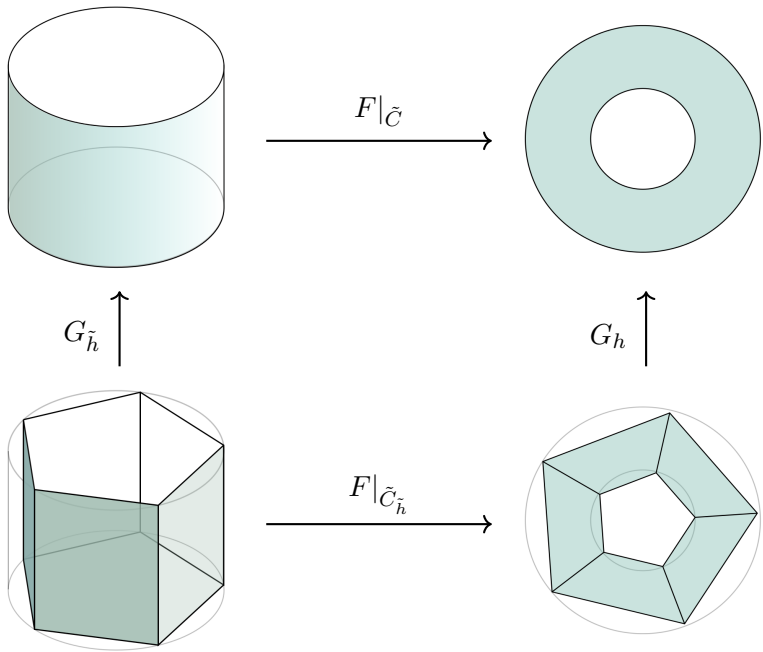

\begin{figure}[h]
\begin{center}
\includegraphics[scale=0.6]{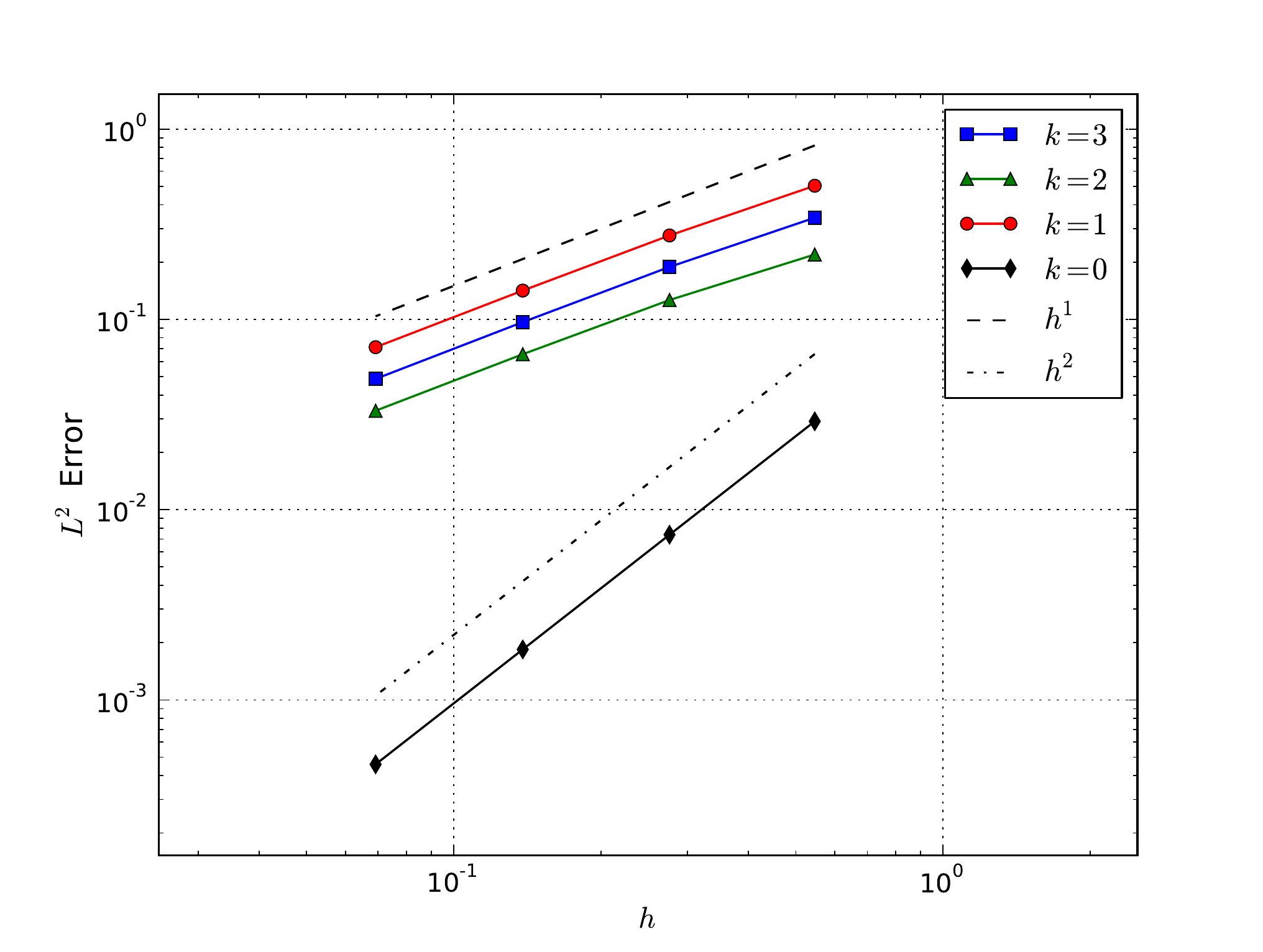}
\caption{Convergence test for the finite element spaces on $C$ defined via pull-back of $\mc E^{-}_{r,s} V^k(\tilde{C}_{\tilde{h}})$ with $r=s=1$. The map $G_h$ is computed using a piecewise linear approximation, which is sufficient to show convergence up to second order. } \label{fig:convtest}
\end{center}
\end{figure}

\begin{figure}[h]
\begin{center}
\includegraphics[scale=0.6]{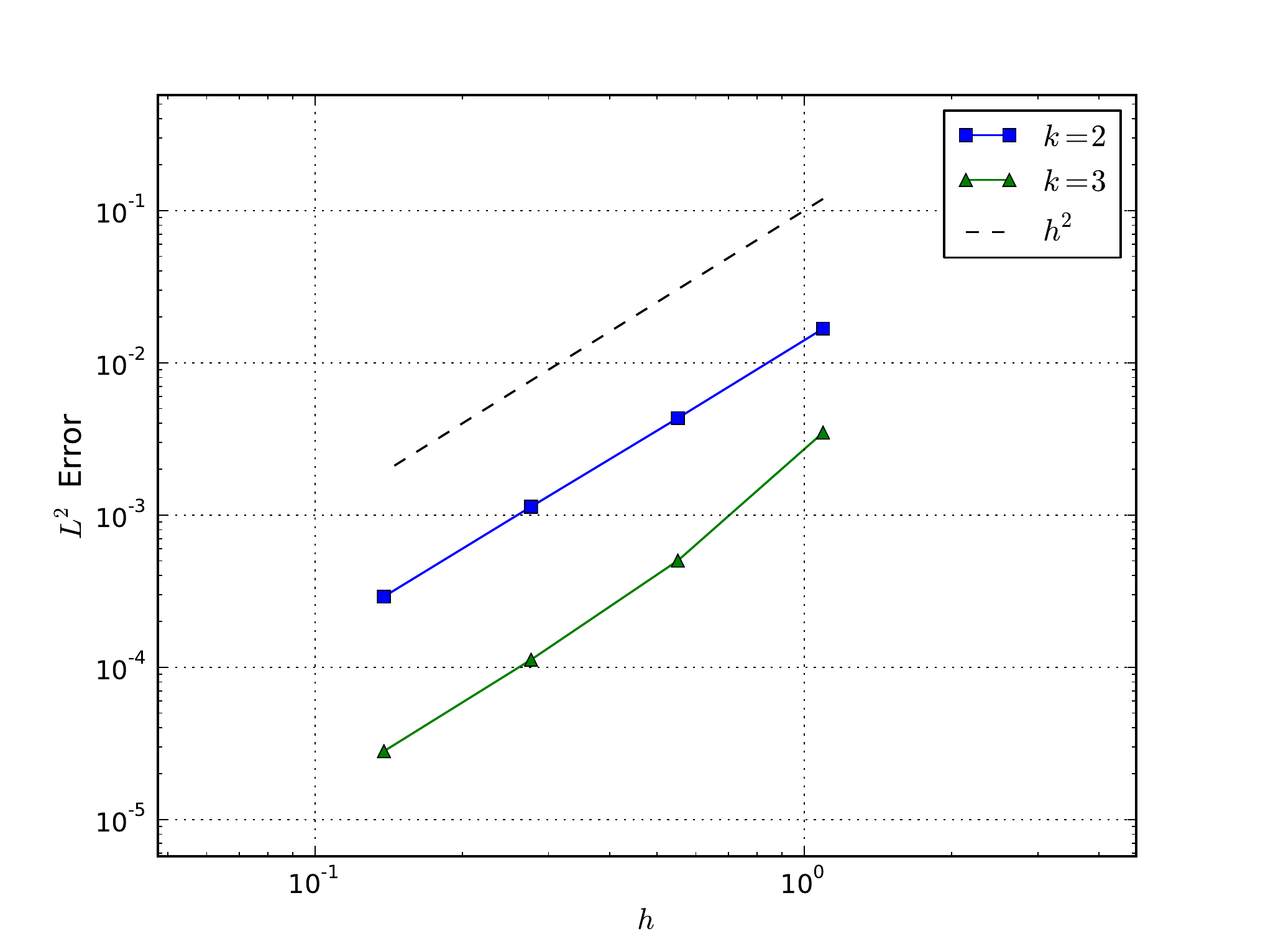}
\caption{Convergence test for them mixed formulation of the Helmholtz equation for 3-forms on $C$, with finite element spaces defined via pull-back of $\mc E^{-}_{r,s} V^k(\tilde{C}_{\tilde{h}})$ with $r=s=2$ and $k=2,3$. The map $G_h$ is computed using a piecewise linear approximation, which is sufficient to show convergence up to second order. } \label{fig:convtesth}
\end{center}
\end{figure}

In general, however, the decomposition of $C$ might be less regular, or in other words it might not be possible to reinterpret it as image of a smooth map $F$. In these cases, one can still rely on the  ``worst-case'' scenario estimates in Theorems~\ref{th:extruded} and \ref{th:extrudedinv}. Note that the results in Theorem~\ref{th:extrudedinv} are particularly relevant for mountain terrain simulation, in which case the computational mesh is usually obtained by modifying in the vertical direction only a decomposition which is otherwise regular in the sense of Theorem~\ref{th:gtransf}. Note however, that smoothing topography is a standard part of numerical weather prediction.

\section{Representation of temperature in compatible finite element spaces}
\label{sec:temp}
There was one piece of crucial information missing from the last
section on three-dimensional finite element spaces: the choice of
finite element space for potential temperature $\theta$. This is the temperature a dry fluid parcel would attain if brought adiabatically to a standard reference pressure $p_0$, typically the pressure at sea-level, i.e.\
\begin{equation}\label{eq:pottemp}
\theta = T \left(\frac{p_0}{p}\right)^{\kappa}\,,
\end{equation}
where $\kappa = R_d/c_p$, $R_d$ is the gas constant of dry air, and $c_p$ is the specific heat capacity at constant pressure.
In this
section, we introduce finite element spaces for $\theta$, and then
describe finite element discretisations for three-dimensional
dynamical cores.

In finite difference dynamical cores, there are two main options for
vertical staggering: the Lorenz grid (temperature collocated with
density), and the Charney-Phillips grid (temperature collocated with
vertical velocity).  The temperature data is always cell centred in
the horizontal.

To mimic the Lorenz grid, we can simply represent temperature in the
density space $V^3_h(\Omega)$. However, many dynamical cores prefer the
Charney-Phillips grid, since it avoids a spurious mode in the
hydrostatic balance equation. To see how to mimic the Charney-Phillips
grid in a compatible finite element discretisation, we need to examine
the structure of $V^2_h(\Omega)$ more closely.

At the level of the reference cell $\hat{K}$, the space $V^2_h(\Omega)$ can be
decomposed into a vertical part $V^{2,v}_h(\Omega)$, and a horizontal part
$V^{2,h}_h(\Omega)$. For example, for a triangular prism, we could take
$V^{2,v}_h(\Omega)=DG_0\otimes CG_1\hat{z}$, where $\hat{z}$ is the unit vector
pointing in the vertical direction. Clearly, all the elements of
$V^{2,v}_h(\Omega)$ have vanishing normal components on the side walls of
$\hat{K}$. Then, a corresponding choice for $V^{2,h}_h(\Omega)$ is $BDM_1\otimes
DG_0$. Since $BDM_1$ is defined on the horizontal triangle, all elements
of $V^{2,h}_h(\Omega)$ point in the horizontal plane and all the elements of
$V^{2,h}_h(\Omega)$ have vanishing normal components on the top and bottom of
$\hat{K}$.  After Piola transformation to a physical cell $K$ with
vertical side walls but possibly sloping top and bottom boundaries, we
see that $V^{2,v}_h(K)$ points in the vertical direction, but now
$V^{2,h}_h(K)$ has a vertical component as well. It is this
non-orthogonality for terrain following grids that leads to pressure gradient errors just as for the C-grid staggered finite difference method.

We see that $V^{2,v}_h(\hat{K})$ is vector-valued, but always points in the vertical direction. We define $V^t_h(\hat{K})$ by
\begin{equation}
  V^t_h(\hat{K}) = \left\{\theta: \theta\hat{z}\in V^{2,v}_h(\hat{K})\right\}.
\end{equation}
Then, given a mapping $F_K:\hat{K}\to K$, we define
\begin{equation}
  V^t_h({K}) = \left\{\theta: \theta\circ F_K\in V^t_h(\hat{K})\right\},
\end{equation}
\emph{i.e.}\ the transformation from the reference element is the usual one for scalar functions, despite $V^t_h(\Omega)$ being constructed from $V^{2,v}_h(\Omega)$.  This defines the extension of the Charney-Phillips staggering to compatible finite element discretisations.

The reason for this construction is that it leads to a one-to-one
mapping between pressure and potential temperature in the hydrostatic pressure equation. To show this, we will describe a compatible finite element discretisation for the three-dimensional compressible Euler equations, and then present results about the hydrostatic balance. 

\subsection{Compatible finite element spatial discretisation for the
three-dimensional compressible Euler equations} 

In this section we present a finite element spatial discretisation for
the three-dimensional compressible Euler equations. Here we are mainly
addressing the question of how to obtain a consistent discretisation
of the various terms when compatible finite element spaces
with only partial continuity are used. Following the methodology of
\cite{wood2014inherently}, we consider the system of equations in
terms of the density $\rho$, and the potential temperature
$\theta$. These variables provide the opportunity for a better
representation of hydrostatic balance in the numerical model.  The equations are
\begin{align}
 u_t + (u\cdot\nabla)u + 2\Omega \times u & = -\theta\nabla \Pi - g\hat{r}, \\
\theta_t + u\cdot\nabla \theta & = 0, \\
\rho_t + \nabla\cdot(u \rho) & = 0,
\end{align}
where $u$ is the velocity, $\Omega$ is the rotation vector in the
Coriolis term, $g$ is the acceleration due to gravity, $\hat{r}$ is the
unit upward vector (which points away from the origin for a global sphere
domain), and $\Pi$ is a function of $\rho$ and $\theta$ given by
\begin{equation}\label{eq:pidef}
\Pi = \left(\frac{R_d \rho\theta}{p_0}\right)^\frac{\kappa}{1-\kappa}.
\end{equation}
Note that within the given formulation, the pressure can be determined via the equation of state $p = \rho R_d T$ combined with the definition of $\Pi$ in Equation~\eqref{eq:pidef} and the one of potential temperature in Equation~\eqref{eq:pottemp}.

To develop the discretisation for the advection term, we use the vector
invariant form
\begin{equation}
(u\cdot \nabla)u = (\nabla\times u)\times u + \nabla \frac{1}{2}|u|^2.
\end{equation}
Then, extending the development of the advection term from the shallow
water equations section, we obtain the weak form
\begin{equation}
\int_\Omega \nabla_h\times(u\times w)\cdot u \diff x - 
\int_\Gamma \{\{n\times (u\times w)\}\}\cdot \tilde{u}\diff S, \quad
\forall w \in V^1_h(\Omega),
\end{equation}
where we adopt the notation that
\begin{equation}
  \{\{n\times F\}\} = n^+\times F^+ + n^-\times F^-,
\end{equation}
where $\nabla_h$ indicates the ``broken'' gradient obtained by
evaluating the gradient pointwise in each cell, and where $\tilde{u}$
is the value of $u$ on the upwind side of each facet.

To develop the discretisation for the pressure gradient term, we
also integrate by parts in each cell to obtain
\begin{equation}
  \int_\Omega \nabla_h\cdot (w\theta)\Pi\diff x 
-\int_{\Gamma_v} \jump{w\theta}\{\Pi\}\diff S, \quad \forall w\in V^1_h(\Omega),
\end{equation}
where $\{\cdot\}$ denotes the average for scalars,
\begin{equation}
  \{f\} = \frac{1}{2}\left(f^++f^-\right).
\end{equation}
where the $\Gamma_v$ is set of all vertically aligned interior facets
(there are no jumps on horizontally aligned facets since
$\jump{w\theta}=0$ there. For lowest order RT elements on cuboids
it can be shown that this discretisation reduces to the finite difference
discretisation currently used in the Met Office Unified Model as described 
in \cite{wood2014inherently}.

Finally, we need to provide discretisations of the transport equations
for $\rho$ and $\theta$. Since $\rho\in V^3_h(\Omega)$ is discontinuous, we can
use the standard upwind Discontinuous Galerkin approach. Building a
transport scheme for $\theta \in V^t_h(\Omega)$ is more delicate, because
$\theta$ is continuous in the vertical. This means that there is no
upwind stabilisation for vertical transport, and so we choose to apply
an SUPG stabilisation there. We develop this in two steps. First we
take the $\theta$ advection equation, multiply by a test function and
integrate by parts, taking the upwind value of $\theta$ on the
vertical column faces, to get
\begin{equation}
\int_\Omega \gamma \theta_t \diff x - \int_\Omega \nabla_h \cdot (u\gamma)\,
\theta \diff x + \int_{\Gamma_v} \jump{\gamma u}\tilde{\theta} \diff S  = 0, 
\quad \forall \gamma \in V^t_h(\Omega).
\end{equation}
Then, to prepare for the SUPG stabilisation, we integrate by parts again,
to give the equivalent formula,
\begin{equation}
\int_\Omega \gamma \theta_t \diff x + \int_\Omega \gamma u\cdot \nabla \theta
\diff x + \int_{\Gamma_v} \jump{\gamma u}\tilde{\theta} - \jump{\gamma u \theta}\diff S  = 0, 
\quad \forall \gamma \in V^t_h(\Omega).
\end{equation}
Since there are now no derivatives applied to $\gamma$, we can obtain an SUPG discretisation by replacing $\gamma \to \gamma + \eta u\cdot\hat{r}\,\Delta t\, \hat{r}\cdot\nabla \gamma$, where $\eta>0$ is a constant upwinding coefficient, to obtain
\begin{align}\nonumber
\int_\Omega \left(\gamma + \eta u\cdot\hat{r}\Delta t \hat{r}\cdot\nabla \gamma\right) \theta_t \diff x + \int_\Omega \left(\gamma + \eta u\cdot\hat{r}\Delta t \hat{r}\cdot\nabla \gamma\right) u\cdot \nabla \theta
\diff x  & \\
\quad + \int_{\Gamma_v}  \jump{\left(\gamma + \eta u\cdot\hat{r}\Delta t \hat{r}\cdot\nabla \gamma\right) u}\tilde{\theta} - \jump{\left(\gamma + \eta u\cdot\hat{r}\Delta t \hat{r}\cdot\nabla \gamma\right) u \theta}\diff S  &= 0, 
\quad \forall \gamma \in V^t_h(\Omega).
\end{align}
This leads to the following spatial discretisation. We seek
$u\in V^2_h(\Omega)$, $\rho,\Pi \in V^3_h(\Omega)$, $\theta \in V^t_h(\Omega)$, such that
\begin{align}
\nonumber
\int_\Omega w\cdot u_t + \nabla\times(u\times w)\cdot u 
 - \int_\Gamma \jump{n\times(u\times w)}\cdot \tilde{u} \diff S & \\
\quad \nonumber
+ \int_\Omega w\cdot 2\Omega\times u
\diff x
- \int_\Omega \nabla_h\cdot (w\theta)\Pi \diff x & \\
\quad
+ \int_\Gamma \jump{w\theta}\{\Pi\}\diff s + \int_\Omega gw\cdot\hat{r}\diff x 
& = 0, \quad \forall w \in V^2_h(\Omega), \\
\nonumber
\int_\Omega \left(\gamma + \eta u\cdot\hat{r}\Delta t \hat{r}\cdot\nabla \gamma\right) \theta_t \diff x + \int_\Omega \left(\gamma + \eta u\cdot\hat{r}\Delta t \hat{r}\cdot\nabla \gamma\right) u\cdot \nabla \theta
\diff x  & \\
\quad + \int_{\Gamma_v}  \jump{\left(\gamma + \eta u\cdot\hat{r}\Delta t \hat{r}\cdot\nabla \gamma\right) u}\tilde{\theta} - \jump{\left(\gamma + \eta u\cdot\hat{r}\Delta t \hat{r}\cdot\nabla \gamma\right) u \theta}\diff S  &= 0, 
\quad \forall \gamma \in V^t_h(\Omega), \\
\int_\Omega \phi \rho_t \diff x - \int_\Omega \nabla_h \phi \cdot \rho \diff x 
+ \int_\Gamma \jump{\phi u}\tilde{\rho}\diff x & = 0, \quad
\forall \phi \in V^3_h(\Omega), \\
\int_\Omega \phi \Pi \diff x - \int_\Omega \phi
\left(\frac{R_d \rho\theta}{p_0}\right)^{\frac{\kappa}{1-\kappa}}\diff x & = 0,
\quad \forall \phi \in V^3_h(\Omega).
\end{align}
Some initial numerical results using this spatial discretisations are provided in Section \ref{sec:slice}.

\subsection{Vertical slice results}\label{sec:slice}
In this section we present some preliminary results obtained using the
discretisations outlined above, in vertical slice $(x,z)$
geometry. The finite element spaces are $(CG_2, RT_1, DG_1)$. The
timestepping scheme uses the semi-implicit formulation described in
Section \ref{sec:semi-implicit} along with a third order, three stage
explicit SSPRK scheme for the advection terms. The tests  were performed using the Firedrake software suite (see \cite{Rathgeber2015}), which allows for symbolic implementation of mixed finite element problems. As the hybridisation
technique is not currently implemented within Firedrake, we instead solve the linear system by eliminating $\theta$, solving the mixed system for $u$ and $p$, then reassembling $\theta$.

The first test case is a nonhydrostatic gravity wave, initially
described in \cite{skamarock1994efficiency} for the Boussinesq
model. The test describes the evolution of inertia gravity waves
excited by a localised, warm $\theta$ perturbation to the background
state of constant buoyancy frequency. Figure \ref{fig:slice} (top)
shows the final $\theta$ perturbation which compares well with results
produced by other models (see for example \cite{melvin2010inherently}).

The second test case describes the evolution of a cold bubble as it
falls, meets the bottom of the domain, spreads and produces
Kelvin-Helmholtz rotors due to the shear instability. The background
state in this case is isentropic. Since the purpose of this test is to
study convergence to a well-resolved solution, viscosity and diffusion
terms are included, the discretisation of which is achieved by using
the interior penalty method. Figure \ref{fig:slice} (bottom) shows the
final state of a well resolved simulation, showing that we capture the
correct position of the front and the details of the instability
\cite{melvin2010inherently}.

Much more comprehensive results in the vertical slice
setting will be presented in a forthcoming paper.

\begin{figure}
  \centerline{
    \includegraphics[width=12cm]{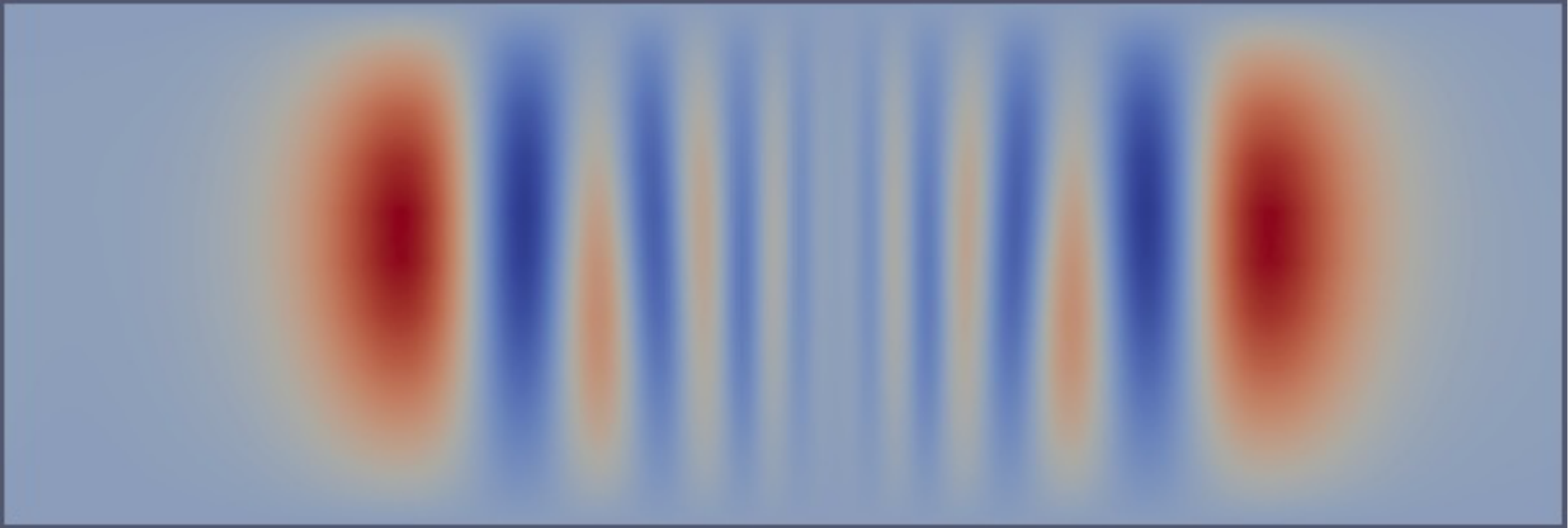}
    }\centerline{
    \includegraphics[width=12cm]{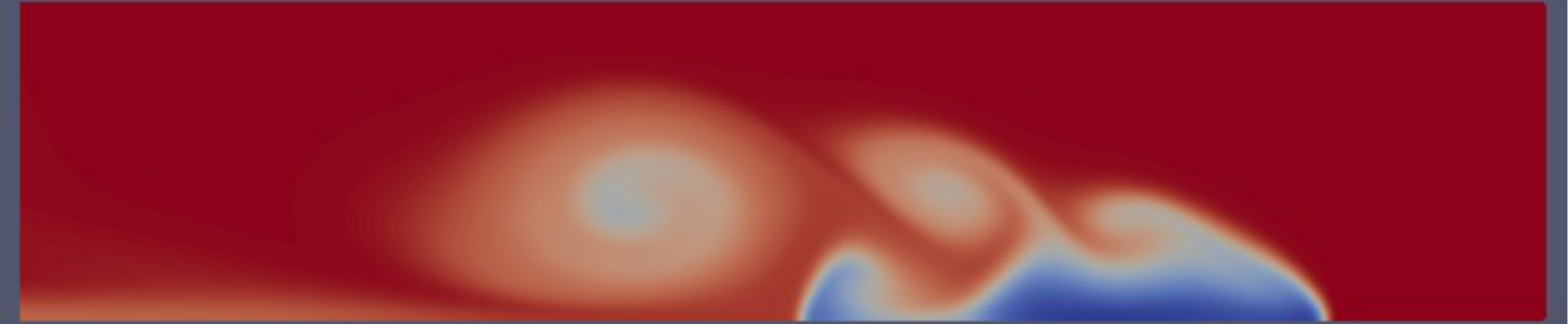}
  }
  \caption{\label{fig:slice}Some initial numerical results using the
    compatible finite element discretisation for the compressible
    Euler equations in a slice geometry (courtesy: Jemma
    Shipton). Top: A potential temperature colour plot showing a
    snapshot from the Skamarock-Klemp gravity wave
    \cite{skamarock1994efficiency} testcase. Bottom: A potential
    temperature colour plot showing a snapshot from the Straka dense
    bubble \cite{straka1993numerical} testcase. }
\end{figure}

\subsection{Hydrostatic balance}
In this section we discuss the hydrostatic balance properties of this
discretisation. It is critical that the basic state of hydrostatic
balance is correctly represented in the numerical model. This means
that there should be a one-to-one correspondence between potential
temperature $\theta$ and $\Pi$ in the state of hydrostatic balance;
this precludes spurious hydrostatic modes from emerging in numerical
simulations close to a hydrostatic state. To obtain the hydrostatic
balance equation, we consider test functions $w$ from the vertical
part $V^{2,v}_h(\Omega)$ of the velocity space $V^2_h(\Omega)$. After neglecting
the advection terms as well as any non-traditional contribution to the
Coriolis force, we obtain
\begin{equation}
  0 = -\int_\Omega \nabla_h\cdot(\theta w)\, \Pi \diff x
  - \int_\Omega w\cdot\hat{r}g \diff x, \quad \forall w \in V^{2,v}_h(\Omega).
\end{equation}
Note that whilst $w$ is written as a vector, it always points in the
$\hat{r}$ direction, since it is in $V^{2,v}_h(\Omega)$. This means that this is
really a scalar equation for the vertical derivative of $\Pi$. It also
explains why the vertical facet integrals have vanished from the
expression. In a domain with level topography, vectors in the
horizontal space $V^{2,h}_h(\Omega)$ do not have a component in the $\hat{r}$
direction, hence no horizontal motion is generated if this equation is
satisfied. However, in terrain-following coordinates, there will be
projection of the pressure gradient of $\Pi$ into the horizontal
direction; this is exactly analogous to the pressure gradient errors observed over topography observed in the C-grid staggered finite difference method.  We also observe that this equation decomposes into
independent equations for each vertical column $c$. We obtain
\begin{equation}
  \label{eq:hydrostatic}
  0 = -\int_c \nabla_h\cdot(\theta w) \Pi \diff x
  - \int_c w\cdot\hat{r}g \diff x, \quad \forall w \in V^{2,v}_h(c),
\end{equation}
for each column $c$ in the mesh. Note here that we are considering
rigid lid boundary conditions; $V^{2,v}_h(\Omega)$ is defined so that $w\cdot
n=0$ for all $w\in V^{2,v}_h(\Omega)$.

We now prove the following two results that show the degree to which
there is a one-to-one correspondence between $\Pi$ and $\theta$
through this equation.
\begin{theorem}
  Let $V^t_h(c)$ be the restriction of $V^t_h(\Omega)$ to column $c$, and let
  $V^{2,v}_h(c)$, $V^{3}(c)$ be similar restrictions for
  $V^{2,v}_h(\Omega)$ an $V^3_h(\Omega)$, respectively. Further, define $\bar{V}^3_h(c)$
  as \begin{equation}
    \bar{V}^3_h(c) = \left\{\rho \in V^3_h(c)\, :\,\int_c \rho \diff x = 0\right\},
  \end{equation}
  i.e.\ the subspace of functions with zero column averages.
  
  For all $\theta\in V^t_h(c)$ satisfying $\theta>0$,
  there is a unique solution $\Pi\in \bar{V}^3_h(c)$ satisfying Equation \eqref{eq:hydrostatic}.
\end{theorem}
\begin{proof}
  First we show that Equation \eqref{eq:hydrostatic} can be
  reformulated as a coupled system of equations. We first note that if
  $v\in {V}^{2,v}_h(c)$, then $\nabla\cdot (\theta v)=\pp{}{r}(\theta
  v)$, where $\pp{}{r}$ is the vertical derivative). Therefore, if
  $\nabla\cdot (\theta v)=0$ then $v=0$, since $v\cdot \hat{r}=0$ at
  the top and bottom of the column. This means that we can add a term
  involving $v$ to Equation \eqref{eq:hydrostatic} without changing
  the solution. Hence, we obtain the equivalent problem of finding
  $v\in V^{2,v}_h(c)$, $\Pi\in V^3_h(c)$ such that
  \begin{align}
    \int_c w\cdot v \diff x + \int_c \nabla\cdot(\theta w) \Pi \diff x
    & = -\int_c gw\cdot\hat{r} \diff x, \quad \forall w \in V^{2,v}_h(c), \\
    \int_c \phi \nabla\cdot (\theta v) \diff x & = 0, \quad \forall \phi \in V^{3}_h(c).
  \end{align}
  This defines a mixed finite element problem in each vertical column, of
  the form
  \begin{align}
    a(w,v) + b(w, \Pi ) & = F(w), \quad \forall w\in V^{2,v}_h(c), \\
    b(v,\phi) & = 0, \quad \forall \phi \in V^3_h(c),
  \end{align}
  with the modified derivative operator $v\mapsto \nabla\cdot(\theta v)$,
  for which Brezzi's stability conditions can be easily verified.
\end{proof}
Similar results are straightforwardly obtained by considering slip
boundary conditions on the bottom surface and constant pressure
boundary conditions on the top surface. 
\begin{theorem}
  Let $\Pi\in\bar{V}^3_h(c)$, and let $\mathring{V}^t_h(c)$ be the subspace
  of $V^t_h(c)$ given by
  \begin{equation}
    \mathring{V}^t_h(c) = \left\{
    \theta \in V^t_h(c):\theta|_{top}=\theta|_{bottom}=0
    \right\},
  \end{equation}
  where $\theta|_{top}$ and $\theta|_{bottom}$ are the values of
  $\theta$ restricted to the top and bottom surface respectively. Let
  $\Pi$ be such that $\hat{r}\cdot\jump{\Pi}>0$ on $\Gamma_c$, where
  $\Gamma_c$ is the set of interior facets in the column $c$. Further,
  if $V^3_h(\Omega)$ is degree $p\geq 1$ in the vertical direction,
  then let $\pp{\Pi}{z}>0$ in each cell.  Then Equation
  \eqref{eq:hydrostatic} has a unique solution for $\theta \in
  \mathring{V}^t_h(c)$.
\end{theorem}
\begin{proof}
  Consider the reference surface mesh from which the extruded mesh
  $\Omega$ was constructed (either a mesh of a planar surface, or a
  mesh approximating the sphere). For a given column $c$, consider the
  base surface cell on the surface mesh corresponding to $c$, and
  construct an affine reference volume cell $\hat{e}$ which is of unit
  height and is extruded orthogonal to the base cell.

  Let $g_e$ be the transformation from $\hat{e}$ into
  each cell $e$ in the column.
  Then, we note that we may write $w\in V^{2,v}_h(c)$ as
  \begin{equation}
    w = J\hat{r}\gamma, \quad \gamma \in \mathring{V}^t_h(c),
  \end{equation}
  where $J$ is the Jacobian defined by
  \begin{equation}
    J|_e = \pp{g_e}{\hat{x}}\circ g^{-1}_e,
  \end{equation}
  in each cell $e$. Then Equation \eqref{eq:hydrostatic} becomes
  \begin{equation}
   A(\gamma,\theta) = -\int_c \pp{}{z}\left(\gamma \theta \kappa
    \right)\Pi \diff x = -\int_c \gamma \kappa g\diff x, \quad \forall \gamma \in V^t_h(c),
  \end{equation}
  where $\kappa=\hat{r}\cdot J\hat{r}$. We assume a non-degenerate
  mesh so that $c_0<\kappa <c_1$ for positive constants $c_0$,
  $c_1$. We then apply integration by parts to obtain
  \begin{equation}
    A(\gamma,\theta) = \int_c \kappa\gamma\theta \pp{}{z}\bigg|_h \Pi \diff x
    + \int_\Gamma \gamma\theta \kappa \hat{r}\cdot \jump{\Pi}\diff S.
  \end{equation}
  Under our assumptions on $\Pi$, $A(\gamma,\theta)$ defines an inner
  product on $\mathring{V}^t_h(c)$, and therefore the solution exists and is unique.
\end{proof}
This result states that $\theta$ is unique up to a choice of values at
the boundary. This is the same situation as we have in the C-grid
staggered finite difference discretisation.
 
\section{Summary and outlook}
\label{sec:sum}

In this survey article, we reviewed recent work on the application of compatible finite element spaces to geophysical fluid dynamics. In particular, we concentrated on analytic results that provide information about the behaviour of these discretisations when applied to geophysically-relevant problems. We also contributed some extra results, including the analysis of inertial modes for the linear shallow water equations, where the conclusion is that there are no spurious inertial modes.

The most substantial contribution is the analysis of the approximation properties of some classes of tensor product finite element spaces that are compatible with the FEEC framework; these spaces have been proposed for the development of global atmosphere models. We applied to such spaces the results of \cite{Arnold14} that state that if the mesh is generated by an arbitrary multilinear transformation one can generally expect a loss in the convergence rate in the $L^2$ norm, which depends on the way the spaces transform under coordinate transformations.
In particular, such convergence loss is more severe as $k$ increases, where $k$ is the order of the space considered as part of the de Rham complex. For example, if the mesh is generated using a prismatic reference element with tensor product polynomial spaces $\mc{E}_{r,r}V^k$, we found that:
\begin{itemize}
\item for $H(\mr{curl})$ elements ($k=1$)  the expected convergence rate is $\lfloor r/2\rfloor$; 
\item for $H(\mr{div})$ elements ($k=2$)  the expected convergence rate is $\lfloor r/2\rfloor-1$;  
\item for $L^2$ elements ($k=3$)  the expected convergence rate is $\lfloor r/2\rfloor-2$. 
\end{itemize}
However, if the mesh is obtained via a sufficiently regular
global transformation as detailed in the statement of
Theorem~\ref{th:gtransf}, one can retrieve optimal convergence
rates independently of $k$. In practice, this implies that tensor product finite element
spaces on prisms, as derived from classical mixed space such as
Raviart-Thomas or Brezzi-Douglas-Marini spaces \cite{Brezzi91}, can
still produce optimal convergence rates even when considered under
non-affine transformations. This is the case, for example, for the
spherical shell problem, although generalisations of the approach
presented here can be easily adapted to more accurate representations
of the Earth surface. 

Finally, we provided some new results analysing
the hydrostatic balance when a particular choice of temperature space
is used, which is compatible with the vertical part of the velocity
space.  These results showed that, under suitable assumptions of
static stability, a hydrostatic pressure can be uniquely determined
from a given temperature profile and vice versa. All of these theoretical
results are underpinning ongoing model development, and detailed numerical
studies and discussion of computational performance will be explored
in future papers.

\bibliographystyle{plain}
\bibliography{ref}

\begin{thebibliography}{10}

\bibitem{arakawa1981potential}
Akio Arakawa and Vivian~R Lamb.
\newblock A potential enstrophy and energy conserving scheme for the shallow
  water equations.
\newblock {\em Mon. Weather Rev.}, 109(1):18--36, 1981.

\bibitem{Arnold13}
Douglas~N Arnold.
\newblock Spaces of finite element differential forms.
\newblock In {\em Analysis and Numerics of Partial Differential Equations},
  pages 117--140. Springer, 2013.

\bibitem{Arnold14}
Douglas~N Arnold, Daniele Boffi, and Francesca Bonizzoni.
\newblock Finite element differential forms on curvilinear cubic meshes and
  their approximation properties.
\newblock {\em Numer. Math.}, 2014.
\newblock arXiv:1204.2595.

\bibitem{Arnold06}
Douglas~N Arnold, Richard~S Falk, and Ragnar Winther.
\newblock Finite element exterior calculus, homological techniques, and
  applications.
\newblock {\em Acta Numer.}, 15:1--155, 2006.

\bibitem{Arnold10}
Douglas~N Arnold, Richard~S Falk, and Ragnar Winther.
\newblock Finite element exterior calculus: from {H}odge theory to numerical
  stability.
\newblock {\em Bull. Amer. Math. Soc. (N.S.)}, 47:281--354, 2010.

\bibitem{arnold2014periodic}
Douglas~N Arnold and Anders Logg.
\newblock Periodic table of the finite elements.
\newblock {\em SIAM News}, 47(9), 2014.

\bibitem{bao2015horizontally}
Lei Bao, Robert Kl{\"o}fkorn, and Ramachandran~D Nair.
\newblock Horizontally explicit and vertically implicit ({HEVI}) time
  discretization scheme for a discontinuous {G}alerkin nonhydrostatic model.
\newblock {\em Monthly Weather Review}, 143(3):972--990, 2015.

\bibitem{boffi2013mixed}
Daniele Boffi, Franco Brezzi, and Michel Fortin.
\newblock {\em Mixed finite element methods and applications}, volume~44.
\newblock Springer, 2013.

\bibitem{brdar2013comparison}
Slavko Brdar, Michael Baldauf, Andreas Dedner, and Robert Kl{\"o}fkorn.
\newblock Comparison of dynamical cores for {NWP} models: comparison of {COSMO}
  and {D}une.
\newblock {\em Theoretical and Computational Fluid Dynamics}, 27(3-4):453--472,
  2013.

\bibitem{brooks1982streamline}
Alexander~N Brooks and Thomas~JR Hughes.
\newblock Streamline upwind/{Petrov-Galerkin} formulations for convection
  dominated flows with particular emphasis on the incompressible
  {Navier-Stokes} equations.
\newblock {\em Computer methods in applied mechanics and engineering},
  32(1):199--259, 1982.

\bibitem{cotter2011numerical}
CJ~Cotter and David~A Ham.
\newblock Numerical wave propagation for the triangular p1 dg--p2 finite
  element pair.
\newblock {\em Journal of Computational Physics}, 230(8):2806--2820, 2011.

\bibitem{cotter2016embedded}
CJ~Cotter and D~Kuzmin.
\newblock Embedded discontinuous galerkin transport schemes with localised
  limiters.
\newblock {\em Journal of Computational Physics}, 311:363--373, 2016.

\bibitem{CoSh2012}
CJ~Cotter and Jemma Shipton.
\newblock Mixed finite elements for numerical weather prediction.
\newblock {\em Journal of Computational Physics}, 231(21):7076--7091, 2012.

\bibitem{Cotter14}
CJ~Cotter and John Thuburn.
\newblock A finite element exterior calculus framework for the rotating
  shallow-water equations.
\newblock {\em Journal of Computational Physics}, 257:1506--1526, 2014.

\bibitem{danilov}
Sergey Danilov.
\newblock Personal Communication.

\bibitem{Da2010}
Sergey Danilov.
\newblock On utility of triangular {C}-grid type discretization for numerical
  modeling of large-scale ocean flows.
\newblock {\em Ocean Dynamics}, 60(6):1361--1369, 2010.

\bibitem{dennis2011cam}
John Dennis, Jim Edwards, Katherine~J Evans, Oksana Guba, Peter~H Lauritzen,
  Arthur~A Mirin, Amik St-Cyr, Mark~A Taylor, and Patrick~H Worley.
\newblock {CAM-SE}: A scalable spectral element dynamical core for the
  {C}ommunity {A}tmosphere {M}odel.
\newblock {\em International Journal of High Performance Computing
  Applications}, page 1094342011428142, 2011.

\bibitem{Brezzi91}
Michel Fortin and Franco Brezzi.
\newblock {\em {Mixed and Hybrid Finite Element Methods (Springer Series in
  Computational Mathematics)}}.
\newblock Springer-Verlag Berlin and Heidelberg GmbH \& Co. K, December 1991.

\bibitem{fournier2004spectral}
Aim{\'e} Fournier, Mark~A Taylor, and Joseph~J Tribbia.
\newblock The spectral element atmosphere model ({SEAM}): High-resolution
  parallel computation and localized resolution of regional dynamics.
\newblock {\em Monthly Weather Review}, 132(3):726--748, 2004.

\bibitem{kelly2013implicit}
Francis~X Giraldo, James~F Kelly, and EM~Constantinescu.
\newblock Implicit-explicit formulations of a three-dimensional nonhydrostatic
  unified model of the atmosphere ({NUMA}).
\newblock {\em SIAM Journal on Scientific Computing}, 35(5):B1162--B1194, 2013.

\bibitem{gopalakrishnan2009convergent}
Jayadeep Gopalakrishnan and Shuguang Tan.
\newblock A convergent multigrid cycle for the hybridized mixed method.
\newblock {\em Numerical Linear Algebra with Applications}, 16(9):689--714,
  2009.

\bibitem{gmd-2015-275}
Jorge~E Guerra and Paul~A Ullrich.
\newblock A high-order staggered finite-element vertical discretization for
  non-hydrostatic atmospheric models.
\newblock {\em Geoscientific Model Development Discussions}, 2016:1--53, 2016.

\bibitem{hiptmair2002finite}
Ralf Hiptmair.
\newblock Finite elements in computational electromagnetism.
\newblock {\em Acta Numerica}, 11:237--339, 2002.

\bibitem{Holst12}
Michael Holst and Ari Stern.
\newblock Geometric variational crimes: Hilbert complexes, finite element
  exterior calculus, and problems on hypersurfaces.
\newblock {\em Foundations of Computational Mathematics}, 12(3):263--293, 2012.

\bibitem{kelly2012continuous}
James~F Kelly and Francis~X Giraldo.
\newblock Continuous and discontinuous {G}alerkin methods for a scalable
  three-dimensional nonhydrostatic atmospheric model: Limited-area mode.
\newblock {\em Journal of Computational Physics}, 231(24):7988--8008, 2012.

\bibitem{marras2015review}
Simone Marras, James~F Kelly, Margarida Moragues, Andreas M{\"u}ller, Michal~A
  Kopera, Mariano V{\'a}zquez, Francis~X Giraldo, Guillaume Houzeaux, and Oriol
  Jorba.
\newblock A review of element-based {G}alerkin methods for numerical weather
  prediction: Finite elements, spectral elements, and discontinuous {G}alerkin.
\newblock {\em Archives of Computational Methods in Engineering}, pages 1--50,
  2015.

\bibitem{marras2013simulations}
Simone Marras, Margarida Moragues, Mariano V{\'a}zquez, Oriol Jorba, and
  Guillaume Houzeaux.
\newblock Simulations of moist convection by a variational multiscale
  stabilized finite element method.
\newblock {\em Journal of Computational Physics}, 252:195--218, 2013.

\bibitem{mcrae2014automated}
Andrew~TT McRae, Gheorge-Teodor Bercea, Lawrence Mitchell, David~A Ham, and
  CJ~Cotter.
\newblock Automated generation and symbolic manipulation of tensor product
  finite elements.
\newblock {\em arXiv preprint arXiv:1411.2940}, 2014.

\bibitem{McRae14}
Andrew~TT McRae and CJ~Cotter.
\newblock {Energy- and enstrophy-conserving schemes for the shallow-water
  equations, based on mimetic finite elements}.
\newblock {\em Quarterly Journal of the Royal Meteorological Society}, 2014.

\bibitem{melvin2010inherently}
Thomas Melvin, Mark Dubal, Nigel Wood, Andrew Staniforth, and Mohamed
  Zerroukat.
\newblock An inherently mass-conserving iterative semi-implicit semi-lagrangian
  discretization of the non-hydrostatic vertical-slice equations.
\newblock {\em Quarterly Journal of the Royal Meteorological Society},
  136(648):799--814, 2010.

\bibitem{melvin2014two}
Thomas Melvin, Andrew Staniforth, and CJ~Cotter.
\newblock A two-dimensional mixed finite-element pair on rectangles.
\newblock {\em Quarterly Journal of the Royal Meteorological Society},
  140(680):930--942, 2014.

\bibitem{natale:_variat_h_finit_elemen_discr}
A~Natale and CJ~Cotter.
\newblock A variational ${H}(\ddiv)$ finite element discretisation for perfect
  incompressible fluids.
\newblock submitted.

\bibitem{Rathgeber2015}
Florian Rathgeber, David~A Ham, Lawrence Mitchell, Michael Lange, Fabio
  Luporini, Andrew~TT McRae, Gheorghe-Teodor Bercea, Graham~R Markall, and
  Paul~HJ Kelly.
\newblock Firedrake: automating the finite element method by composing
  abstractions.
\newblock {\em To appear in ACM TOMS}, 2016.

\bibitem{Ringler10}
Todd~D Ringler, John Thuburn, Joseph~B Klemp, and Willam~C Skamarock.
\newblock {A unified approach to energy conservation and potential vorticity
  dynamics for arbitrarily-structured C-grids}.
\newblock {\em Journal of Computational Physics}, 229(9):3065--3090, 2010.

\bibitem{rognes2013automating}
Marie~E Rognes, David~A Ham, CJ~Cotter, and Andrew~TT McRae.
\newblock Automating the solution of pdes on the sphere and other manifolds in
  fenics 1.2.
\newblock {\em Geoscientific Model Development}, 6(6):2099--2119, 2013.

\bibitem{rognes2009efficient}
Marie~E Rognes, Robert~C Kirby, and Anders Logg.
\newblock Efficient assembly of h(div) and h(curl) conforming finite elements.
\newblock {\em SIAM Journal on Scientific Computing}, 31(6):4130--4151, 2009.

\bibitem{rostand2008raviart}
Virgile Rostand and Daniel~Y Le~Roux.
\newblock {Raviart--Thomas and Brezzi--Douglas--Marini finite-element
  approximations of the shallow-water equations}.
\newblock {\em {International Journal for Numerical Methods in Fluids}},
  57(8):951--976, 2008.

\bibitem{shipton:_higher}
Jemma Shipton and CJ~Cotter.
\newblock Higher-order compatible finite element schemes for the nonlinear
  rotating shallow water equations on the sphere.
\newblock Submitted.

\bibitem{skamarock1994efficiency}
Willam~C Skamarock and Joseph~B Klemp.
\newblock Efficiency and accuracy of the {K}lemp-{W}ilhelmson time-splitting
  technique.
\newblock {\em Monthly Weather Review}, 122(11):2623--2630, 1994.

\bibitem{staniforth2013analysis}
Andrew Staniforth, Thomas Melvin, and CJ~Cotter.
\newblock Analysis of a mixed finite-element pair proposed for an atmospheric
  dynamical core.
\newblock {\em Quarterly Journal of the Royal Meteorological Society},
  139(674):1239--1254, 2013.

\bibitem{StTh2012}
Andrew Staniforth and John Thuburn.
\newblock Horizontal grids for global weather and climate prediction models: a
  review.
\newblock {\em Q. J. Roy. Met. Soc}, 138(662A):1--26, 2012.

\bibitem{straka1993numerical}
Jerry~M Straka, Robert~B Wilhelmson, Louis~J Wicker, John~R Anderson, and
  Kelvin~K Droegemeier.
\newblock Numerical solutions of a non-linear density current: A benchmark
  solution and comparisons.
\newblock {\em International Journal for Numerical Methods in Fluids},
  17(1):1--22, 1993.

\bibitem{thomas2005ncar}
Stephen~J Thomas and Richard~D Loft.
\newblock The {NCAR} spectral element climate dynamical core: Semi-implicit
  eulerian formulation.
\newblock {\em Journal of Scientific Computing}, 25(1):307--322, 2005.

\bibitem{thuburn2015primal}
John Thuburn and CJ~Cotter.
\newblock A primal--dual mimetic finite element scheme for the rotating shallow
  water equations on polygonal spherical meshes.
\newblock {\em Journal of Computational Physics}, 290:274--297, 2015.

\bibitem{walters1998robust}
Roy~A Walters and Vincenzo Casulli.
\newblock A robust, finite element model for hydrostatic surface water flows.
\newblock {\em Communications in Numerical Methods in Engineering},
  14(10):931--940, 1998.

\bibitem{wood2014inherently}
Nigel Wood, Andrew Staniforth, Andy White, Thomas Allen, Michail Diamantakis,
  Markus Gross, Thomas Melvin, Chris Smith, Simon Vosper, Mohamed Zerroukat,
  et~al.
\newblock An inherently mass-conserving semi-implicit semi-lagrangian
  discretization of the deep-atmosphere global non-hydrostatic equations.
\newblock {\em Quarterly Journal of the Royal Meteorological Society},
  140(682):1505--1520, 2014.

\end{thebibliography}

\end{document}